\UseRawInputEncoding
\documentclass[11pt]{amsart}
\usepackage{nccmath}

\usepackage{amsmath,amsthm,amssymb}
\usepackage{times}
\usepackage{enumerate}
\usepackage{color}
\usepackage{accents}
\newcommand{\ubar}[1]{\underaccent{\bar}{#1}}

\newtheorem{prop}{Proposition}
\newtheorem{theo}[prop]{Theorem}
\newtheorem{lemm}[prop]{Lemma}

\newtheorem{claim}{Claim}

\newcommand{\be}{\begin{equation}}
\newcommand{\ee}{\end{equation}}
\newcommand{\lt}{\left}
\newcommand{\rt}{\right}
\newcommand{\al}{\alpha}
\newcommand{\e}{\epsilon}
\renewcommand{\leq}{\leqslant}
\renewcommand{\geq}{\geqslant}
\newcommand{\td}{\tilde}

\newcommand{\ka}{\kappa}

\newcommand{\s}{\sigma}
\newcommand{\R}{\mathbb{R}}
\newcommand{\M}{\mathcal{M}}
\newcommand{\bx}{\bar{x}}

\newcommand{\goto}{\rightarrow}
\newcommand{\dS}{\mathbb{S}}
\newcommand{\us}{u^*}
\newcommand{\uj}{u^J}
\newcommand{\ujs}{u^{J*}}

\newcommand{\vjs}{\varphi^{J*}}

\newcommand{\lus}{\ubar{u}^*}

\newcommand{\lu}{\ubar{u}}

\newcommand{\gjs}{g^{J*}}
\newcommand{\uu}{\bar{u}}

\newcommand{\gas}{\gamma^*}

\newcommand{\T}{\partial}
\newcommand{\p}{\partial}
\newcommand{\hp}{\hat{p}}
\newcommand{\pM}{\psi^{M_1}}

\newcommand{\w}{w^*}
\newcommand{\z}{\mathbf{z}}
\newcommand{\tz}{\tilde{\mathbf{z}}}
\newcommand{\hz}{\hat{\mathbf{z}}}
\newcommand{\B}{\mathcal{B}}
\newcommand{\F}{\mathcal{F}}
\newcommand{\Chi}{\mbox{\Large$\chi$}}
\newcommand{\conv}{\text{Conv}}

\numberwithin{equation}{section}

\begin{document}
\setlength{\baselineskip}{1.2\baselineskip}

\title[Constant Hessian curvature hypersurface in the Minkowski space]
{entire spacelike constant $\sigma_k$ curvature hypersurfaces with prescribed boundary data at infinity}

\author{Zhizhang Wang}
\address{School of Mathematical Science, Fudan University, Shanghai, China}
\email{zzwang@fudan.edu.cn}
\author{Ling Xiao}
\address{Department of Mathematics, University of Connecticut,
Storrs, Connecticut 06269}
\email{ling.2.xiao@uconn.edu}
\thanks{ Research of the first author is sponsored by Natural Science Foundation of Shanghai, No.20JC1412400, 20ZR1406600 and supported by NSFC Grants No.11871161.}

\begin{abstract}
 In this paper, we investigate the existence and uniqueness of convex, entire, spacelike hypersurfaces of constant
$\s_k$ curvature with prescribed set of lightlike directions $\F\subset\dS^{n-1}$ and perturbation $q$ on $\F$. We prove that given a closed set $\F$ in the ideal boundary at infinity of hyperbolic space and a perturbation $q$ that satisfies some mild conditions, there exists a complete entire spacelike constant $\sigma_k$ curvature hypersurface $\M_u$ with prescribed set of lightlike directions $\F$ satisfying when $\frac{x}{|x|}\in\F,$ as $|x|\goto\infty,$ $u(x)-|x|\goto q\lt(\frac{x}{|x|}\rt).$
This result is new even for the case of constant Gauss curvature. We also prove that when the Gauss map image is a half disc $\bar{B}_1^+$ and the perturbation $q\equiv 0,$  if a CMC hypersurface $\M_u$ satisfies $|u(x)-V_{\bar{\mathcal{B}}_+}(x)|$ is bounded, then $u(x)$ is unique.

\end{abstract}

\maketitle

\section{Introduction}
Let $\R^{n, 1}$ be the Minkowski space with the Lorentzian metric
\[ds^2=\sum_{i=1}^{n}dx_{i}^2-dx_{n+1}^2.\]
In this paper, we study convex spacelike hypersurfaces with positive constant
$\sigma_k$ curvature in Minkowski space $\R^{n, 1}$. Here, $\s_k$ is the $k$-th elementary symmetric polynomial, i.e.,
 \[\s_k(\ka)=\sum\limits_{1\leq i_1<\cdots<i_k\leq n}\ka_{i_1}\cdots\ka_{i_k}.\]
Any such hypersurface can be written locally as a graph of a function
$x_{n+1}=u(x), x\in\R^n,$ satisfying the spacelike condition
\be\label{int1.1}
|Du|<1.
\ee

Treibergs started the research of constructing nontrivial entire spacelike CMC hypersurfaces in \cite{Tre}. He showed that
for any $q\in C^2(\dS^{n-1}),$ there is a spacelike, convex, CMC hypersurface $\M_u=\{(x, u(x))\mid x\in\R^n\}$ with bounded principal curvatures, such that as $|x|\goto\infty$,
 $u(x)\goto |x|+q\lt(\frac{x}{|x|}\rt)$. The result in \cite{Tre} was generalized by Choi-Treibergs in \cite{CT}, where they proved that
 for any closed set $\F\subset\dS^{n-1}$ and $q\in C^0(\F),$ there is a spacelike convex CMC hypersurface $\M_u,$ such that when $\frac{x}{|x|}\in \F,$
  $u(x)\goto |x|+q\lt(\frac{x}{|x|}\rt),$ as $|x|\goto\infty.$ We will call $q$ the perturbation of $u$ on lightlike directions.

One natural question to ask is: can we construct convex entire spacelike constant
$\s_k$ curvature hyersurfaces with prescribed set of lightlike directions $\F\subset\dS^{n-1}$ and an arbitrary $C^0$ perturbation on $\F$?

There have been many challenges to answering this question. Here are some partial results.
Li \cite{Li} extended the result in \cite{Tre} to constant Gauss curvature. He proved that for any $q\in C^2(\dS^{n-1}),$ there is a spacelike constant Gauss curvature hypersurface $\M_u$ with bounded principal curvatures, such that as $|x|\goto\infty$,
 $u(x)\goto |x|+q\lt(\frac{x}{|x|}\rt)$. In 2006, Guan-Jian-Schoen \cite{GJS} showed that if
 $\F=\p B_1^+\cap\dS^{n-1}=\dS^{n-1}\cap\{x_1\geq 0\}$ and $q\in C^\infty (\bar{B}_1^+)$ satisfying $q|_{\p_0B_1^+}$ is affine, then
 there is a spacelike constant Gauss curvature hypersurface $\M_u$ such that when $\frac{x}{|x|}\in\dS^{n-1}_+,$
 $u(x)\goto |x|+q\lt(\frac{x}{|x|}\rt)$ as $|x|\goto\infty$. Here $B_1^+=\{x\mid |x|<1\,\,,x_1>0\}$ and $\p_0B_1^+=\p B_1^+\cap \{x_1=0\}.$  Very recently, under the same settings as in \cite{Tre} and \cite{Li}, Ren-Wang-Xiao\cite{RWX} and Wang-Xiao \cite{WX} solved the existence problem for constant $\s_{k}$ curvature hypersurfaces. More precisely, for any $q\in C^2(\dS^{n-1}),$ they constructed an entire spacelike, strictly convex, constant $\s_k$ curvature hypersurface $\M_u$ with bounded principal curvatures, which satisfies as $|x|\goto\infty$,
 $u(x)\goto |x|+q\lt(\frac{x}{|x|}\rt).$

 \subsection{Main results}
 \label{subintr1}
 In this paper, we will investigate the existence and uniqueness of convex, entire, spacelike hypersurfaces of constant
$\s_k$ curvature with prescribed lightlike directions $\F$ and perturbation $q$ on $\F$. Our main results are the following.
\begin{theo}
\label{intthm1}
Suppose $\mathcal{F}\subset\dS^{n-1}$ is the closure of an open subset with $\T\F\in C^{1,1}$ and $q\in C^{2, 1}(\dS^{n-1})\cap C^{2,1}_0(\F)$ is a set of $C^{2,1}$ function satisfying $q\equiv 0$ on $\dS^{n-1}\setminus\F.$  Then for $1<k\leq n$, there exists a smooth,
entire, spacelike, strictly convex hypersurface $\M_u=\{(x,u(x))\mid x\in\R^n\}$ satisfying
\be\label{int1.0}
\sigma_k(\ka[\M_u])=\binom{n}{k},
\ee
 where $\ka[\M_u]=(\kappa_1,\kappa_2,\cdots,\kappa_n)$ is the principal curvatures of $\M_u$.
 Moreover, when $\frac{x}{|x|}\in\F,$
 \be\label{int1.0'}
 u(x)\goto |x|+q\lt(\frac{x}{|x|}\rt),\,\mbox{as $|x|\goto\infty.$}
 \ee
Furthermore, the Gauss map image of $\M_u$ is the convex hull $\conv(\F)$ of $\F$ in the unit disc.
\end{theo}

We want to point out here the condition $q\in C^{2, 1}(\dS^{n-1})\cap C^{2,1}_0(\F)$ is needed to construct the sub and supersolution
of equation \eqref{int1.0} that satisfy condition \eqref{int1.0'}. In particular, we require $q\in C^{2, 1}$ instead of $q\in C^2$ to make sure
that, as $|x|\goto\infty$ the supersolution $\uu$ is greater than or equal to the subsolution $\lu$ (for details see Lemma \ref{sub-tricky direction lem}
and Lemma \ref{subsuper-comparison lem}). We also note that Theorem 4.1 in \cite{GJS} is a special case of Theorem \ref{intthm1}.

In 1976, Cheng-Yau \cite{CY} proved the only entire maximal hypersuface ($\sigma_1=0$) is a hyperplane. Very recently, Hong-Yuan \cite{HY} obtained the
precise asymptotic behavior of maximal hypersurfaces over exterior domains. However, there is no known uniqueness result for the constant $\s_k$ curvature spacelike hypersurface with prescribed set of lightlike directions $\F$ and perturbation $q$. This problem is open even for constant mean curvature hypersurfaces. The main difficulty is that we only know the asymptotic behavior of $u(x)$ at infinity in the lightlike directions, and only have some mild information in other directions. Our second theorem is a uniqueness result for CMC hypersurfaces under certain conditions. In particular, we prove
\begin{theo}
\label{intthm2}
Let $\M_u$ be an $n$-dimensional CMC hypersurfaces satisfying following conditions:

(1) $\sigma_1(\ka[\M_u])=n;$

(2) When $\frac{x}{|x|}\in\bar{\mathcal{B}}_+$ or $x_1\goto-\infty, |\bar{x}|\goto\infty,$ we have
$u(x)-V_{\bar{\mathcal{B}}_+}(x)\goto 0$ as $|x|\goto\infty$;

(3) $|u-V_{\bar{\mathcal{B}}_+}(x)|\leq C_0$ for any $x\in\R^n.$

Then $u$ is the standard semitrough $\z^1.$ Here $\bar{\B}_+=\dS^{n-1}\cap\{x_1\geq 0\}$ and $\bar{x}=(x_2, \cdots, x_n).$
\end{theo}

\subsection{Similarities, differences and difficulties of the current paper vs \cite{WX}}
\label{subintr2}
In this subsection, we will first review the ideas in \cite{WX}, then we will analyze the main differences and difficulties of ideas in this paper.
\subsubsection{Ideas and results in \cite{WX}}
\label{sub1.2.1}
Recall that in \cite{WX}, instead of solving equation \eqref{int1.0} directly, we studied the following problem
\be\label{main equation-ball}
\left\{
\begin{aligned}
F(\w\gas_{ik}\us_{kl}\gas_{lj})&=\frac{1}{\binom{n}{k}^{\frac{1}{k}}}\,\,\text{in $\td{F}$},\\
\us&=0\,\,\text{on $\F,$}
\end{aligned}
\right.
\ee
where $\w=\sqrt{1-|\xi|^2},$ $\gas_{ij}=\delta_{ij}-\frac{\xi_i\xi_j}{1+\w},$ $\us_{kl}=\frac{\p^2u}{\p\xi_k\p\xi_l},$ $\F\subset\dS^{n-1}$ as described in Theorem \ref{intthm1},
$\td{F}$ is the convex hull of $\F$ in $B_1:=\{\xi\mid |\xi|<1\},$ and $F(\w\gas_{ik}\us_{kl}\gas_{lj})=\lt(\frac{\s_n}{\s_{n-k}}(\ka^*[\w\gas_{ik}\us_{kl}\gas_{lj}])\rt)^{1/k}.$
Here, $\ka^*[\w\gas_{ik}\us_{kl}\gas_{lj}]=(\ka^*_1, \cdots, \ka^*_n)$ are the eigenvalues of the matrix $(\w\gas_{ik}\us_{kl}\gas_{lj}).$
By Subsection 2.3 and Lemma 15 of \cite{WX}, we know that if $u^*$ solves \eqref{main equation-ball} then the Legendre transform of $u^*,$ denoted by
$u,$ satisfies \eqref{int1.0}. Moreover, as $|x|\goto\infty,$ $u(x)\goto |x|+0$ when $\frac{x}{|x|}\in\F.$

Since \eqref{main equation-ball} is degenerate,
we considered the following approximating problems
\be\label{dirichlet-ball}
\left\{
\begin{aligned}
F(\w\gas_{ik}\ujs_{kl}\gas_{lj})&=\frac{1}{\binom{n}{k}^{\frac{1}{k}}}\,\,\text{in $\td{F}_J$}\\
\us&=\vjs\,\,\text{on $\p\td{F}_J,$}
\end{aligned}
\right.
\ee
where $\{\td{F}_J\}_{J=1}^\infty$ is a sequence of smooth convex set in $\td{F}$ that approaches $\td{F},$
$\vjs=\lus|_{\p\td{F}_J},$ and $\lus$ is the Legendre transform of the subsolution $\lu$ of \eqref{int1.0}. The solvability of
\eqref{dirichlet-ball} was proved in Section 5 of \cite{WX}. Thus, we obtained a sequence of solutions $\{u^{J*}\}.$
Considering the Legendre transform of $\{u^{J*}\}$ which we denote by $\{u^J\},$ then $\{u^J\}$ is a sequence of functions satisying
\eqref{int1.0} on bounded domain $\Omega_J\subset\R^n.$ Moreover, $\{\Omega_J\}$ is an exhausting sequence of sets on $\R^n.$
In Section 6 of \cite{WX}, we showed that $\{u_J\}$ converges to the desired entire solution $u$ of \eqref{int1.0}, which satisfies
when $\frac{x}{|x|}\in\F,$ as $|x|\goto\infty,$ $u(x)\goto |x|+0$. We want to point out that the hard part of the proof of the convergence is to obtain the local $C^1$ estimates (see Subsection 6.2 in \cite{WX}). In order to establish the local $C^1$ estimates,
by Lemma 22 of \cite{WX}, for any compact set $K\subset\R^n,$ we need to construct a spacelike cutoff function $\Psi_K$ such that $\Psi_K>\uu$ when
$|x|\goto\infty$ and $\psi_K<\lu$ when $x\in K,$ where $\uu$ is a supersolution of \eqref{int1.0}.

\subsubsection{Differences and difficulties of the current problem}
\label{sub1.2.2}
Let's take a closer look at the details of the process described in Subsubsection 1.2.1. Note that we require the domain $\{\Omega_J\}$ of $\{u^J\}$
to be an exhausting sequence of sets on $\R^n.$ This condition is equivalent to saying
\be\label{boundary-condition}
\lt|Du^{J*}\rt|_{\p\td{F}_J}\goto\infty\,\, \mbox{as $J\goto\infty$,}
\ee
which is highly nontrivial. This is why in equation \eqref{dirichlet-ball} we have to choose the boundary value $\vjs$ to be $\lus|_{\p\td{F}_J}$. Notice that when $\lu$ is a subsolution of \eqref{int1.0}, then its Legendre transform $\lus$ is a supersolution of \eqref{dirichlet-ball}.
In view of the Comparison Theorem, it's easy to see that this choice of $\vjs$ guarantees the condition \eqref{boundary-condition} to be met. Therefore, we can see that the existence of a strictly convex subsolution of equation \eqref{int1.0} is necessary for us.

On the other hand, when we look at the sequence $\{u^J\},$
it's not hard to see that in order to control the behavior of the limit function $u$ at infinity we also need a supersolution of \eqref{int1.0}. Moreover,
since we require $u\goto |x|+q\lt(\frac{x}{|x|}\rt)$ as $|x|\goto\infty$ for $\frac{x}{|x|}\in\F,$ our $\uu$ and $\lu$ have to satisfy
$\uu,\lu\goto |x|+q\lt(\frac{x}{|x|}\rt)$  as $|x|\goto\infty$ for $\frac{x}{|x|}\in\F$ as well.

In the case when $q\equiv 0$ (the case discussed in \cite{WX}), by \cite{BS} we know that
there exist constant Gauss curvature hypersurfaces with prescribed set of lightlike directions $\F$ and $0$ perturbation; by \cite {CT} we know that there
exist CMC hypersurfaces with prescribed set of lightlike directions $\F$ and $0$ perturbation. It's not difficult to see that we can use these two results to construct our sub and supersolution. The tricky part there is to construct spacelike cutoff functions $\Psi_K,$ which is necessary to obtain local $C^1$ estimates (for details see Subsection 6.2 of \cite{WX}).

In this paper, we want to consider the case when $q\neq 0.$ The first obstruction is that we don't have known subsolutions anymore. Therefore, we need to construct  subsolutions from scratch. Secondly, in order to construct $\Psi_K,$ we need a better control for the barrier function $\uu$ in the non-lightlike directions $\dS^{n-1}\setminus\F$. It's much harder to control the asymptotic behavior of $\uu$ for the case when $q\neq 0$ than the case when $q\equiv0$.

\subsection{Outline}
\label{outline}
Let's assume we have a strictly convex subsolution $\lu$ and a supersolution $\uu$ such that $\uu,\lu\goto |x|+q\lt(\frac{x}{|x|}\rt)$  as $|x|\goto\infty$ for $\frac{x}{|x|}\in\F$. Moreover, we also assume for every compact set $K\subset \R^n$ there exists a spacelike cutoff function $\Psi_K$ satisfying  $\Psi_K>\uu$ when $|x|\goto\infty$ and $\psi_K<\lu$ when $x\in K.$ Following the steps described in Subsubsection \ref{sub1.2.1}, we can obtain our desired entire solution satisfying \eqref{int1.0} and \eqref{int1.0'}. Therefore, in this paper, we will focus on the constructions of $\lu,$ $\uu,$ and $\Psi_K.$

In Section \ref{se}, adapting the idea of \cite{CT},  we construct a strictly convex entire solution to equation \eqref{int1.0}, whose Gauss map image is
the half disk $\bar{B}_1^+:=\{\xi\in\R^n\mid |\xi|\leq 1, \xi_1\geq 0\}.$ We denote this solution by $\z^k$ and call it semitrough. By applying Lorentz transform to $\z^k,$ we know for any closed geodesic ball
$\bar{\B}\subset\dS^{n-1},$ there exists a strictly convex entire solution to equation \eqref{int1.0}, whose Gauss map image is $\conv(\bar{\B}).$
We denote this solution by $\z^k_{\bar{\B}}.$

In Section \ref{properties of semitroughs}, we list some basic properties of semitroughs that are needed in later sections.

In Section \ref{construction}, inspired by the ideas of \cite{CT} and \cite{BS}, we denote
\[\ubar{\z}^k_{\F}=\sup\limits_{\bar{\B}\subset\F, \delta(\bar{\B})\geq\delta_0}\z^k_{\bar{\B}}(x),\]
and
\[\td{\z}^k_{\F}(x, y)=q(y)-M+\ubar{\z}^k_{\F}(x+p(y)),\]
where $p(y)=Dq(y)+My.$ Then the subsolution is chosen as
\[\lu^k_{\F}(x)=\sup\limits_{y\in\dS^{n-1}}\td{\z}^k_{\F}(x, y).\]
 One can check that $\lu^k_{\F}(x)\goto V_{\F}(x)+q\lt(\frac{x}{|x|}\rt)$ as $|x|\goto\infty$ for $\frac{x}{|x|}\in\F$, where $V_{\F}(x)=\sup\limits_{\lambda\in\F}x\cdot\lambda.$

The construction of the supersolution is more delicate. Unlike \cite{CT}, in this paper, we need to construct the spacelike cutoff function $\Psi_K,$ which demands a better control of $\uu^k_{\F}(x)$ at infinity. Therefore, our construction of $\uu^k_{\F}(x)$ is different from the one in \cite{CT}. Moreover,
as stated in Theorem \ref{intthm1}, we require that the perturbation $q\in C^{2,1}(\dS^{n-1})\cap C^{2,1}_0(\F)$ to be a $C^{2,1}$ function satisfying $q\equiv 0$ on $\dS^{n-1}\setminus\F,$ which is more restrictive than \cite{CT}. In particular,
let \[\bar{\z}^k_{\F}(x)=\inf\limits_{\bar{\B}\supset\F, \delta(\bar{\B})\leq\pi-\delta_0}\z^k_{\bar{\B}}(x)\]
and \[\hz^k_{\F}(x, y)=q(y)+MV_\F(y)+\bar{\z}^k_{\F}(x+\hp(y)),\]
where $\hp(y)=Dq(y)-My.$ Then the supersolution is chosen as
\[\uu^k_{\F}(x)=\inf\limits_{y\in\bar{\B}_{\delta}\lt(\frac{x}{|x|}\rt)}\hz^k_{\F}(x, y).\]
Here, $\delta>0$ is so small such that for any$\frac{x}{|x|}\perp(\p\td{F}\cap B_1)$ and $\frac{x}{|x|}\notin\F,$
$\bar{\B}_{\delta}\lt(\frac{x}{|x|}\rt)\cap\F=\emptyset.$ Here $\td{F}$ is the convex hull of $\F$ in $B_1.$

The sub and supersolutions we obtained above satisfy $\uu^k_{\F}(x)-\lu^k_{\F}(x)\goto 0$ as $|x|\goto\infty$ in almost all directions.
This enables us to construct the spacelike cutoff function $\Psi_K$ (for details see Subsections \ref{constructcut} and \ref{uls}).

Recall that in Subsubsection \ref{sub1.2.1} we pointed out that to solve this problem we will need the Legendre transform of the subsolution $\lu.$
Despite the subsolution $\lu^k_{\F}(x)$ we constructed in Section \ref{construction} is strictly convex, it's not differentiable. In Section \ref{ltbf}, we show that the Legendre transform of $\lu^k_{\F}(x),$ denoted by $\lus$ is however in $C^1(\td{F}).$

Unfortunately, in order to solve the Dirichlet problem \eqref{dirichlet-ball}, the boundary value $\varphi^{J*}=\lus|_{\p\td{F}_J}$ need to be a $C^{1, 1}$ function at least. Therefore, in Section \ref{cs}, we study the properties of the inf-convolution of $\lus.$ Combining the results in Sections \ref{construction}, \ref{ltbf}, and \ref{cs}, together with results in \cite{WX}, we prove Theorem \ref{intthm1}.

From our construction of subsolutions in Section \ref{construction}, we can see that, in fact we can find more than one subsoltion to equation
\eqref{int1.0} that satisfies the desired asymptotic behavior at infinity. One may ask: do different subsolutions give us different solutions? In other words, can we construct more than one solution to equation \eqref{int1.0} that satisfies $u(x)\goto |x|+q\lt(\frac{x}{|x|}\rt)$ for $\frac{x}{|x|}\in\F$ as $|x|\goto\infty$? In Section \ref{uniqueness}, we discuss about this question and prove Theorem \ref{intthm2}.

\bigskip
\section{Semitrouph}
\label{se}
In $\R^{n, 1},$ the spacelike hypersurfaces invariant under $SO(n-1, 1)$ action are described by a single function $f(x_1).$
Denoting $\bx=(x_2, \cdots, x_n)$, then the height function is $u(x_1, \cdots, x_n)=\sqrt{f(x_1)^2+|\bx|^2}.$
Let $\M_u=\{(x, u(x))\mid x\in\R^n\}$ be the hypersurface determined by $u.$ By a straightforward calculation we get,
$\ka[\M_u],$ the principal curvatures of $\M_u$ are
\[\ka_1=\frac{f''}{(1-f')^{3/2}},\,\,\ka_2=\cdots=\ka_n=\frac{1}{f(1-f'^2)^{1/2}}.\]
Since $\s_k(\ka[\M_u])=\s_k(\ka|\ka_1)+\ka_1\s_{k-1}(\ka|\ka_1)=\binom{n}{k},$ the function $f$ satisfies the following
ODE
\be\label{se1.1}
\frac{kf''}{f^{k-1}(1-f'^2)^{k/2+1}}+\frac{n-k}{f^k(1-f'^2)^{k/2}}=n.
\ee
Following \cite{CT} we can prove
\begin{lemm}
\label{selem1}
There is a solution $f(t)$ of \eqref{se1.1} defined for all $t\in\R$ with the following properties:

(1) $0<f'<1$ and $f''>0$ for all $t\in\R.$

(2) $\lim\limits_{t\goto-\infty}f(t)=\lt(\frac{n-k}{n}\rt)^{1/k}:=l_k$.

(3) $\max\{l_k, t\}<f(t)<\sqrt{1+t^2}$ for all $t.$

(4) There is a constant $C>0$ depending on $n$ and $k$ so that $|f(t)-\sqrt{1+t^2}|<Ct^{-n-1}$
whenever $t>0$ large.
\end{lemm}
\begin{proof}
This proof is a modification of the proof of Lemma 5.1 in \cite{CT}, for readers' convenience, we will include it here.
First we change variables and let $u(y)=f'(t),$ $y=f(t),$ then we have $f''=uu'.$ Thus, \eqref{se1.1}
can be written as
\be\label{se1.2}
\frac{kuu'}{y^{k-1}(1-u^2)^{k/2+1}}+\frac{n-k}{y^k(1-u^2)^{k/2}}=n.
\ee
 Multiplying $y^{n-1}$ on both sides gives
 \[\frac{kuu'y^{n-k}}{(1-u^2)^{k/2+1}}+\frac{(n-k)y^{n-k-1}}{(1-u^2)^{k/2}}=ny^{n-1},\]
 which yields
 \[\lt[y^{n-k}(1-u^2)^{-k/2}\rt]'=ny^{n-1}.\]
 Integrating both sides from $y_0:=f(0)$ to $y$ with initial condition $u(y_0)=f'(0)=0,$ we get
 \be\label{se1.3}
 (1-u^2)^{-k/2}=y^k+\frac{y_0^{n-k}(1-y_0^k)}{y^{n-k}}.
 \ee
 Now, note that $l_k=\lt(\frac{n-k}{n}\rt)^{1/k}$ is a constant solution of \eqref{se1.1}.
 If we choose $l_k<y_0<1,$ then \eqref{se1.3} shows that, there exits some $0\leq u<1$ and $u$ is an increasing function which
 exists for all $y\geq y_0.$ Moreover, it's easy to see that when $y>y_0,$ $f{''}=uu'>0.$ Therefore, we find a solution
 $f$ of \eqref{se1.1} satisfies $0<f'<1,$ $f''>0$ for $t>0$, $f(0)>l_k,$ $f'(0)=0,$
 and $f$ exists for all $t\geq 0.$

 Next, we will construct a sequence $\{f_i\}$ that satisfy $0<f_i'<1,$ {\bf $f''_i>0,$
 $f_i(a_i)>l_k,$ $f'_i(a_i)=0,$} and $f_i$ exists for all $t\geq a_i.$ Moreover, we want {\bf $y_i:=f_i(a_i)\goto l_k$}
 and $a_i\goto -\infty$ as $i\goto+\infty.$ Then we will show that $\{f_i\}$ converges to a function $f$ that satisfies
 $(1), (2), (3),$ and $(4).$

 Let's choose a decreasing sequence $l_k<y_i<1$ such that $y_i\goto l_k.$ Let $f_{[\al, i]}$ denote the solution of
 \eqref{se1.1} with the initial value condition $f_{[\al, i]}(\al)=y_i$ and $f'_{[\al, i]}(\al)=0.$ Furthermore, we require
 $f_{[\al, i]}$ is defined on $[\al, \infty).$ Since $f_{[1, i]}(1)<\sqrt{2}=\sqrt{1+t^2}\big|_{t=1},$
 $f'_{[1, i]}(1)=0<\frac{1}{\sqrt{2}}=(\sqrt{1+t^2})'\big|_{t=1},$ and $\sqrt{1+t^2}$ also satisfies equation \eqref{se1.1}.
 By the Comparison Theorem, we know that $f_{[1, i]}(t)<\sqrt{1+t^2}$ for all $t\geq 1.$ Thus, we may define
 \[a_i=\inf\{\al\leq 1: \sqrt{1+t^2}>f_{[\al, i]}(t), \mbox{for $\forall t\geq \al$}\}.\]
 Let $\{f_i\}$ be the sequence of solutions $f_i(t)=f_{[a_i, i]}(t).$ In the following we will prove $a_i\goto-\infty$ and
 $f_i$ converges to the desired solution $f.$

 \begin{claim}$a_i\goto-\infty$ as $i\goto\infty.$\end{claim}

 To see this, observe that $f_i$ and $\sqrt{1+t^2}$ doesn't have interior contact point. Due to the choice of $a_i$ we know that the contact point is at
 $t\goto\infty.$ Moreover, we have $f_i'<1,$ thus
 \[f_i(1)>\lim\limits_{t\goto\infty}[\sqrt{1+t^2}-(t-1)]=1,\]
 and if $a_i\leq 0,$ then $f_i(0)<1.$

 Now, let $y=y_i+z,$ by the Taylor expansion, \eqref{se1.3} implies
 \be\label{se1.4}
 \begin{aligned}
 1+\frac{k}{2}u^2&\leq 1+\lt[ky_i^{k-1}-\frac{(n-k)(1-y_i^k)}{y_i}\rt]z\\
 &+\lt[\binom{k}{2}y^{k-2}_i+\frac{(1-y_i^k)(n-k)(n-k+1)}{2y_i^2}\rt]z^2+Mz^3,
 \end{aligned}
 \ee
 where $M=M(k, n, y_i)>0$ and $y_i\leq y\leq 1.$ For our convenience, let's denote
 \[\gamma:=ky_i^{k-1}-\frac{(n-k)(1-y_i^k)}{y_i}\,\,\text{and}\,\, \beta:=\binom{k}{2}y_i^{k-2}+\frac{(1-y_i^k)(n-k)(n-k+1)}{2y^2_i},\]
 then \eqref{se1.4} can be written as
 \[1+\frac{k}{2}u^2\leq 1+\gamma z+\beta z^2+Mz^3.\]
 Taking $0<c_i\leq 1$ which satisfies $f_i(c_i)=1,$ and let $b_i(y)=f_i^{-1}(y).$ Applying
 \eqref{se1.4} we get
 \be\label{se1.5}
\begin{aligned}
c_i-b_i(y)&=\int_y^1\frac{dy}{u(y)}\geq\int_{y-y_i}^{1-y_i}\frac{dz}{\sqrt{\frac{2}{k}(\gamma z+\beta z^2+Mz^3)}}\\
&\geq \int_{y-y_i}^{1-y_i}\frac{1}{\gamma+M_2 z}dz\\
&=\frac{1}{M_2}\log\left\{\frac{\gamma+M_2(1-y_i)}{\gamma+M_2(y-y_i)}\right\}.
\end{aligned}
\ee
Here $M_2>0$ is chosen such that $\frac{2}{k}(\gamma z+\beta z^2+Mz^3)<(\gamma+M_2z)^2.$
Now let $a_i=b_i(y_i)$ and note that as $y_i\goto l_k,$ $\gamma\goto 0.$ Therefore,
we have $c_i-a_i\goto\infty.$ This proves the Claim 1.

Notice that, as mentioned above, $f_i$ and $\sqrt{1+t^2}$ don't have interior contact point, and hence, they are
asymptotic. This means that $f'\goto 1$ as $t\goto \infty.$ Since $0<f'_i<1,$ it's easy to see that $f_i(t)>\max\{l_k, t\}$
for $t\in[a_i, \infty).$

\begin{claim} There exists a subsequence of $\{f_i\}$ that converges to a function $f$ that satisfies $(1), (2),$ and $(3).$ \end{claim}
By what we have shown we know
\[\max\{l_k, t\}<f_i(t)<\sqrt{1+t^2}\,\,\text{for}\,\,t\in(a_i, \infty),\]
where $a_i\goto-\infty$ as $i\goto\infty.$

We will choose a subsequence $\{f_{i_l}\}$ such that $a_{i_{l+1}}<a_{i_l}.$
Note that for any interval $[b, d]\subset\R,$ there exists $N_b>0$ such that
whenever $l>N_b,$ we have $a_{i_l}<b$ and
\[\max\{l_k, b\}\leq f_{i_l}(t)\leq\sqrt{1+d^2}\,\,\text{on}\,\,[b, d].\]
Moreover, by \eqref{se1.3} we also have for $l>N_b,$
\[f_{i_l}'(t)\leq C(b, d, l_k)<1\,\,\text{on}\,\,[b, d].\]
Thus we conclude that $\{f_{i_l}\}\goto f$ uniformly on compact set and $f$ satisfies $(1), (2),$ and $(3).$

Finally we will prove that $f$ satisfies $(4).$ First, note that when $y_0=1$ the solution to \eqref{se1.3} is the hyperboloid
$b(t)=\sqrt{1+t^2}$ with corresponding $v(y)=b'(t)$ such that
\be\label{se1.6}
\frac{1}{\sqrt{1-v^2}}=y=b(t).
\ee
Since $f(t)$ is the solution to \eqref{se1.3} with $y_0=l_k$ we get,
\[\frac{d}{dy}\lt(b^{-1}(y)-f^{-1}(y)\rt)=\frac{1}{v}-\frac{1}{u}.\]
For our convenience, let's denote $\e_k:=\frac{l_k^{n-k}-l_k^n}{y^{n-k}},$ then we have
\be\label{se1.7}
\begin{aligned}
\frac{d}{dy}\lt(b^{-1}(y)-f^{-1}(y)\rt)&=\frac{1}{\sqrt{1-y^{-2}}}-\frac{1}{\sqrt{1-(y^k+\e_k)^{-2/k}}}\\
&=\frac{1-\lt(1+\frac{\e_k}{y^k}\rt)^{-2/k}}{y^2\sqrt{1-y^{-2}}{\sqrt{1-(y^k+\epsilon_k)^{-2/k}}}\cdot\lt(\sqrt{1-y^{-2}}+\sqrt{1-(y^k+\e_k)^{-2/k}}\rt)}.
\end{aligned}
\ee
Thus, when $y>0$ large we derive
\[\frac{d}{dy}\lt(b^{-1}(y)-f^{-1}(y)\rt)\leq \frac{C}{y^{n+2}}.\]
Since $f$ is asymptotic to $b$ at $\infty,$ we obtain
\[\lim\limits_{y\goto\infty}b^{-1}(y)-f^{-1}(y)=0.\] Now let $z=b^{-1}(y)$ then we conclude that
\be\label{se1.8}
\begin{aligned}
0&<b(z)-f(z)\leq f^{-1}(y)-b^{-1}(y)\\
&\leq\int_y^\infty\lt(\frac{1}{v}-\frac{1}{u}\rt)dy\leq\frac{C}{y^{n+1}}\leq\frac{C}{z^{n+1}}.
\end{aligned}
\ee
This completes the proof of Lemma \ref{selem1}.
\end{proof}
\bigskip
\bigskip

\section{The properties of Semitroughs}
\label{properties of semitroughs}
So far we obtained an entire function $\z^k$ given by $\z^k(x)=\sqrt{f_k(x_1)^2+|\bx|^2},$ whose graph
has constant $\s_k$ curvature, and
\[D\z^k{\bf(\mathbb{R}^n)}=\lt(\frac{f_kf_k'}{\z}, \frac{\bx}{\z}\rt)=\{\xi\in B_1:\xi_1>0\}:=B_1^+,\]
where $B_1:=\{\xi\in\R^n\mid |\xi|<1\}.$
Moreover, it's easy to see that, for $\theta=(\theta_1,\theta_2,\cdots, \theta_n)\in\mathbb{S}^{n-1}$,
\be\label{cb1.1}
\lim\limits_{r\goto\infty}{\left(\z^k( r\theta)-V_{\bar{B}_1^+}( r\theta)\right)}=
\left\{
\begin{aligned}
&l_k,\,\,\theta\bot\T_0\bar{B}_1^+:=\{\xi\in\bar{B}_1, \xi_1=0\}\text{ and } \theta_1=-1\\
&0,\,\,\text{elsewhere,}
\end{aligned}
\right.
\ee
here following the notation of \cite{CT}, $V_{\bar{E}}(x)=\sup\limits_{\lambda\in\bar{E}}x\cdot\lambda,$ for any
$E\subseteq \bar{B}_1$.
In the following, let's denote by $d_S$ the natural distance on $\dS^{n-1}.$ For any $x, y\in\dS^{n-1}$ we have
\[d_S(x, y)=\arccos(x\cdot y)\in[0, \pi],\]
where the dot stands for the canonical scalar product in $\R^n.$ A ball in $\dS^{n-1}$ is a ball in the metric
space $(\dS^{n-1}, d_S),$ i.e., a set
\[\B=\{x\in\dS^{n-1}\mid d_S(x, x_0)<\delta\},\]
where $x_0\in\dS^{n-1}$ and $\delta>0$ is the radius of $\B,$ also denoted by $\delta(\B).$

Applying Lorentz transformation to $\z^k$
\be\label{cb1.2}
\left\{
\begin{aligned}
x_1'&=\frac{x_1-\al x_{n+1}}{\sqrt{1-\al^2}}\\
x_i'&=x_i\\
x'_{n+1}&=\frac{x_{n+1}-\al x_1}{\sqrt{1-\al^2}},
\end{aligned}
\right.
\ee
we get $\td{\z}^k(x_1',\cdots, x_n')=x_{n+1}'=\frac{\z^k-\al x_1}{\sqrt{1-\al^2}}.$

By a straightforward calculation we obtain, for $i\geq 2,$
\[0=\frac{\T x_1}{\T x_i'}-\al\frac{\T \z^k}{\T x_1}\frac{\T x_1}{\T x_i'}-\al\z^k_i,\]
this gives us
\[\frac{\al\z^k_i}{1-\al\z^k_1}=\frac{\T x_1}{\T x_i'};\]
while for $i=1$
\[\sqrt{1-\al^2}=(1-\al\z^k_1)\frac{\T x_1}{\T x_1'},\]
which implies
\[\frac{\T x_1}{\T x_1'}=\frac{\sqrt{1-\al^2}}{1-\al\z^k_1}.\]

Therefore, we have
\[\frac{\p\td{\z}^k}{\p x_1'}=\frac{(\z^k_1-\al)}{\sqrt{1-\al^2}}\cdot\frac{\sqrt{1-\al^2}}{1-\al\z^k_1}=\frac{\z^k_1-\al}{1-\al\z^k_1},\]
and for $i\geq 2,$
\begin{align*}
\frac{\p\td{\z}^k}{\p x_i'}&=\frac{1}{\sqrt{1-\al^2}}\lt(\z^k_1\frac{\T x_1}{\T x_i'}+\z^k_i-\al\frac{\T x_1}{\T x_i'}\rt)\\
&=\frac{1}{\sqrt{1-\al^2}}\lt[\frac{\al(\z^k_1-\al)\z^k_i}{1-\al\z^k_1}+\z^k_i\rt]\\
&=\frac{\z^k_i\sqrt{1-\al^2}}{1-\al\z^k_1}.\\
\end{align*}
This yields
\[D\td{\z}^k({\bf\mathbb{R}^n})=\lt(\frac{\z^k_1-\al}{1-\al\z^k_1}, \sqrt{1-\al^2}\frac{\z^k_i}{1-\al\z^k_1}\rt):=\{\xi\in B_1, \xi_1>-\al\}.\]
From the above calculations we can see that, for every ball $\bar{\B}$ of $\dS^{n-1},$ there exists an entire function
$\tilde{\z}^k$ defined on $\R^n,$ whose graph is a hypersurface with constant $\s_k$ curvature and the Gauss map image of such hypersurface is the convex
hull of $\bar{\B}.$ We denote this function by $\z^k_{\bar{\B}}$.

The next lemma gathers properties of semitroughs that we will use to construct the barriers (for semitroughs of constant Gauss curvature see \cite{BS}).
\begin{lemm}\label{cblem1}
Let $\bar{\B}$ be a closed ball of $\dS^{n-1}$ such that $\pi-\delta_0\geq\delta(\B)\geq\delta_0>0$ and let
$\z^k_{\bar{\B}}$ be the corresponding semitrough with $\s_k$ curvature equals $\binom{n}{k},$ where $1\leq k\leq n.$
Then the following holds:

(1) Let $g(x)=\sqrt{1+|x|^2},$ $\td{B}$ be the convex hull of $\bar{\B}$ in $B_1,$
and $\p_0\td{B}=\p\td{B}\cap B_1.$ Then
 \be\label{cb1.3}
g\geq \z^k_{\bar{\B}}>V_{\bar{\B}}=V_{\td{B}}.
\ee
Moreover, let $x=r\theta,$ for any fixed $\theta\in\mathbb{S}^{n-1}$, as $r\goto\infty$ we have
\be\label{cb1.4}
\left\{
\begin{aligned}
&\z^k_{\bar{\B}}-V_{\bar{\B}}\goto \frac{l_k}{\sqrt{1-\alpha^2}},\,\,
\text{when {\bf $\theta\notin\bar{\B}$} is perpendicular to $\p_0\td{B},$ }\\
&\z^k_{\bar{\B}}-V_{\bar{\B}}\goto 0,\,\,\text{otherwise,}\\
\end{aligned}
\right.
\ee
where $-1<\alpha<1$ depends on $\delta(\bar{\B}),$
$V_{\bar{\B}}(x)=\sup\limits_{\xi\in\bar{\B}}\xi\cdot x,$ and $V_{\td{B}}(x)=\sup\limits_{\xi\in\td{B}}\xi\cdot x.$

(2) For all compact sets $K\subset \R^n$ there exists $\delta=\delta(K, \delta_0, k, n)>0$ such that for all $x\in K$
\be\label{cb1.6}
\z^k_{\bar{\B}}(x)\geq V_{\bar{\B}}(x)+\delta.
\ee

(3) For all compact sets $K\subset\R^n$ there exists $\upsilon_K=\upsilon(K, \delta_0, k, n)\in (0, 1]$ such that
for all $x, y\in K$
\be\label{cb1.7}
|\z^k_{\bar{\B}}(x)-\z^k_{\bar{\B}}(y)|\leq (1-\upsilon_K)|x-y|.
\ee

(4) Let $\z^1_{\bar{\B}}$ denote the semitrough with $\s_1$ curvature equals $n$, and $\z^k_{\bar{\B}}$ denote the semitrough with
$\s_k$ curvature equals $\binom{n}{k}.$ Then $\z^1_{\bar{\B}}>\z^k_{\bar{\B}}.$
\end{lemm}
\begin{proof}
Notice that, since $g(x)$ is invariant under the Lorentzian transform, we only need to look at these assertions for the standard semitrough $\z^k_{\bar{\B}_+}$.
For part (1) to part (3), the constant Gauss curvature case has been proved in \cite{BS}.

We first prove (1). When $\frac{x}{|x|}\in\bar{\B}_+:=\lt\{\xi\mid|\xi|=1, \xi_1\geq 0\rt\},$ by Lemma \ref{selem1} we have
\begin{align*}
\z^k_{\bar{\B}_+}(x)-V_{\bar{\B}_+}(x)&=\sqrt{f^2_k(x_1)+|\bar{x}|^2}-|x|\\
&=\frac{f^2_k(x_1)-x_1^2}{\sqrt{f^2_k(x_1)+|\bar{x}|^2}+|x|}>0.
\end{align*}
When $\frac{x}{|x|}\notin\bar{\B}_+,$ we have
\[\z^k_{\bar{\B}_+}(x)-V_{\bar{\B}_+}(x)=\frac{f^2_k(x_1)}{\sqrt{f^2_k(x_1)+|\bar{x}|^2}+|\bar{x}|}>0.\]
It's easy to see that equations \eqref{cb1.3} and \eqref{cb1.4} follow through directly.

Part (2) can be derived from part (1); part (3) is due to $\z^k_{\bar{\B}_+}$ is spacelike. Thus, we only need to show part (4).
Let $\z^1_{\bar{\B}_+}(x)=\sqrt{f_1^2(x_1)+|\bx|^2},$ then by the proof of Lemma 5.1 in \cite{CT}, we know $f_1$ is the solution of
\[\frac{f_1''}{(1-f_1'^2)^{3/2}}+\frac{(n-1)}{f_1(1-f_1'^2)^{1/2}}=n.\]
Let $\z^k_{\bar{\B}_+}(x)=\sqrt{f^2_k(x_1)+|\bx|^2},$ then by Maclaurin's inequality and the proof of Lemma \ref{selem1} we get
\[\frac{kf_1''}{f_1^{k-1}(1-f_1'^2)^{k/2+1}}+\frac{n-k}{f_1^k(1-f_1'^2)^{k/2}}\leq n
 =\frac{kf_k''}{f_k^{k-1}(1-f_k'^2)^{k/2+1}}+\frac{n-k}{f_k^k(1-f_k'^2)^{k/2}}.\]
Moreover, in view of Lemma 5.1 of \cite{CT} and Lemma \ref{selem1}, we also have
$\lim\limits_{t\goto-\infty}(f_1(t)-f_k(t))=\lt(1-\frac{1}{n}\rt)-(1-\frac{k}{n})^{1/k}>0,$
and $\lim\limits_{t\goto\infty}(f_1(t)-f_k(t))=0.$
By the Comparison Theorem we conclude that $f_1(t)>f_k(t)$ for all $t.$
This completes the proof of part (4).
\end{proof}

\bigskip

\section{Constructions of barrier functions and spacelike cutoff function}
\label{construction}
\setcounter{equation}{0}
In this section, we will construct the strictly convex subsolution, supersolution, and spacelike cutoff function that satisfy the desired asymptotic behavior at infinity. We will first construct them for the case when the prescribed Gauss map image is the half disc
$\bar{B}_1^+:=\{\xi\in\R^n\mid |\xi|\leq 1, \xi_1\geq 0\}.$ Then we will construct them for general prescribed Gauss map image $\conv(\F).$
\subsection{Construction of the subsolution}
\label{constructionsub}
In the following, we first consider $\F=\bar{\B}_+.$ We also assume the boundary perturbation $q$ satisfies the following conditions.
\be\label{q.1}
q\in C^{2,1}_{\text{loc}}\lt(\R^n\setminus\{0\}\rt)\,\,\mbox{and $q(x)=q\lt(\frac{x}{|x|}\rt)$},
\ee
\be\label{q.2}
q(x)\equiv 0\,\,\mbox{when $x\in \dS^{n-1}\setminus\B_+.$}
\ee
By \eqref{q.1} we know there exists $a_2>0$ such that for any $x, y\in\dS^{n-1}$ we have,
\be\label{q.3}
|q(x)-q(y)-Dq(y)\cdot(x-y)|\leq a_2|x-y|^2=-2a_2y\cdot(x-y).
\ee

Now for $y\in \dS^{n-1},$ by \eqref{q.1} it's easy to see that
$y\cdot Dq(y)=0.$
Let $p(y)=Dq(y)+My,$ where $M>2a_2$ is a constant, and
\begin{align*}
\tz^k_{\bar{\B}_+}(x, y)&=q(y)-M+\z^k_{\bar{\B}_+}(x+p(y))\\
&=q(y)-M+\sqrt{f_k^2(x_1+p_1)+\sum\limits_{i=2}^n(x_i+p_i)^2}.\\
\end{align*}
Here $\z^k_{\bar{\B}_+}$ is the standard semitrough with $\s_k$ curvature equals $\binom{n}{k}$, and $f_k$ is the solution of \eqref{se1.1}
that satisfies the properties (1)--(4) in Lemma \ref{selem1}.

\begin{lemm}
\label{addlem4.1}
 If as $|x|\goto\infty$, $V_{\bar{\B}_+}(x)\goto\infty,$ then we have
$\sup\limits_{y\in\dS^{n-1}}\tz^k_{\bar{\B}_+}(x, y)-V_{\bar{\B}_+}(x)\goto q\lt(\frac{x}{|x|}\rt).$
\end{lemm}
\begin{proof} We will divide this proof into 3 cases.

\textbf{Case 1.} When $\frac{x}{|x|}\in\B_+,$ without loss of generality we assume
\[x=(r\cos\beta, 0, \cdots, 0, r\sin\beta)=r\cos\beta e_1+r\sin\beta e_n\]
and $\cos\beta>0.$ We also assume $\sin\beta\geq 0.$ Then as $r\goto\infty$ we get,
\be\label{sub1.1}
\begin{aligned}
&\tz^k_{\bar{\B}_+}(x, y)-V_{\bar{\B}_+}(x)\\
&\goto q(y)-M+\sqrt{1+(r\cos\beta+p_1)^2+(r\sin\beta+p_n)^2+\sum\limits_{i=2}^{n-1}p_i^2}-r\\
&\goto q(y)-M+\cos\beta p_1+\sin\beta p_n\\
&=q(y)-M+\lt<\frac{x}{|x|}, p(y)\rt>:=G(y)
\end{aligned}
\ee
By our assumption \eqref{q.3} we have
\begin{align*}
&q\lt(\frac{x}{|x|}\rt)-q(y)\\
&\geq 2a_2y\cdot\lt(\frac{x}{|x|}-y\rt)+Dq(y)\cdot\lt(\frac{x}{|x|}-y\rt).
\end{align*}
Therefore, we obtain
\be\label{sub1.2}
\begin{aligned}
&G\lt(\frac{x}{|x|}\rt)-G(y)\\
&=q\lt(\frac{x}{|x|}\rt)-q(y)+M-\lt<\frac{x}{|x|}, Dq(y)+My\rt>\\
&\geq 2a_2y\cdot\lt(\frac{x}{|x|}-y\rt)+M\lt(1-\frac{x}{|x|}\cdot y\rt)\\
&=(M-2a_2)\lt(1-\frac{x}{|x|}\cdot y\rt)\geq 0.
\end{aligned}
\ee

\textbf{Case 2.} When $\frac{x}{|x|}\in\p\B_+,$ without loss of generality we assume
$x=(0, \cdots, 0, r)=re_n.$ Then as $r\goto\infty$ we have,
\be\label{sub1.3}
\begin{aligned}
&\tz^k_{\bar{\B}_+}(x, y)-V_{\bar{\B}_+}(x)\\
&=q(y)-M+\sqrt{f_k^2(p_1)+(r+p_n)^2+\sum\limits_{i=2}^{n-1}p_i^2}-r\\
&\goto q(y)-M+\lt<e_n, p(y)\rt>:=G(y).
\end{aligned}
\ee
Same as before we can see that
\be\label{sub1.4}
\begin{aligned}
&G(e_n)-G(y)\\
&=q(e_n)-q(y)+M-\lt<e_n, Dq(y)+My\rt>\\
&\geq 2a_2y\cdot(e_n-y)+Dq(y)\cdot(e_n-y)+M-\lt<e_n, Dq(y)+My\rt>\\
&=(M-2a_2)(1-e_n\cdot y)\geq 0.
\end{aligned}
\ee

\textbf{Case 3.} When $\frac{x}{|x|}\notin\bar{\B}_+$ and $V_{\bar{\B}_+}\lt(\frac{x}{|x|}\rt)\neq 0,$ i.e., $\frac{x}{|x|}$ is not perpendicular to
$\p_0B_1^+.$ Without loss of generality we assume
$x=(r\cos\beta, 0, \cdots, 0, r\sin\beta)$ where $\cos\beta<0$ and $\sin\beta>0.$
Then as $r\goto\infty$ we have,
\be\label{sub1.5}
\begin{aligned}
&\tz^k_{\bar{\B}_+}(x, y)-V_{\bar{\B}_+}(x)\\
&\goto q(y)-M+\sqrt{l_k^2+(r\sin\beta+p_n)^2+\sum\limits_{i=2}^{n-1}p_i^2}-r\sin\beta\\
&\goto q(y)-M+\lt<e_n, p(y)\rt>:=G(y).
\end{aligned}
\ee
From the discussion in \eqref{sub1.4} we have
$G(y)\leq G(e_n).$

Note that when $V_{\bar{\B}_+}\lt(\frac{x}{|x|}\rt)\goto0$ and $V_{\bar{\B}_+}(x)\goto\infty$ as $|x|\goto\infty$ follow the same discussion of case 3, we can show
that $G(y)\leq G(e_n).$ This completes the proof of Lemma \ref{addlem4.1}.
\end{proof}

Now consider the case when $V_{\bar{\B}_+}\lt(\frac{x}{|x|}\rt)\goto0$ and $V_{\bar{\B}_+}(x)<\infty$ as $|x|\goto\infty.$ Without loss of generality we assume
$x=(x_1, 0, \cdots, 0, x_n)$ where $x_1\goto-\infty$ and $x_n\geq 0$ is bounded.
As $|x|\goto\infty,$ we have
\be\label{sub1.6}
\begin{aligned}
&\tz^k_{\bar{\B}_+}(x, y)-V_{\bar{\B}_+}(x)\\
&\goto q(y)-M+\sqrt{l_k^2+(x_n+p_n)^2+\sum\limits_{i=2}^{n-1}p_i^2}-x_n\\
&\leq C,
\end{aligned}
\ee
where $C=C(|q|_{C^1}, n, k, M)>0$ is a bounded constant.

Let
$\lu^k_{\bar{\B}_+}(x)=\sup\limits_{y\in\mathbb{S}^{n-1}}\tz^k_{\bar{\B}_+}(x, y),$ it's easy to see that
$\lu^k_{\bar{\B}_+}$ is a subsolution to the equation $\sigma_k(\ka[\M_u])={n\choose k}$. Moreover, from above discussions we conclude that
 when $V_{\bar{\B}_+}(x)\goto\infty$ as $|x|\goto\infty,$ we have $\lu^k_{\bar{\B}_+}(x)-V_{\bar{\B}_+}(x)\goto q\lt(\frac{x}{|x|}\rt).$  We need to understand the asymptotic behavior of $\lu^k_{\bar{\B}_+}(x)$ as $|x|\goto\infty$ when $V_{\bar{\B}_+}(x)<\infty.$

\begin{lemm}
\label{sub-tricky direction lem}
If as $|x|\goto\infty,$ $V_{\bar{\B}_+}\lt(\frac{x}{|x|}\rt)\goto0$ and $V_{\bar{\B}_+}(x)<\infty,$ then we have
$\sup\limits_{y\in\dS^{n-1}}\td{\z}^k_{\bar{\B}_+}(x, y)\leq-M+\sqrt{l_k^2+(V_{\bar{\B}_+}(x)+M)^2.}$
\end{lemm}
\begin{proof}
Without loss of generality, let's assume $x=(x_1, 0, \cdots, 0, x_n),$ where $x_1\goto-\infty$ and $x_n\geq 0$ is bounded.
In view of \eqref{q.1} and \eqref{q.2}, we know that there exists $a_0, a_1>0$ such that
\be\label{std1}
|q(y)|\leq a_0y_1^2, |Dq(y)|\leq a_1|y_1|\,\,\mbox{for $y\in\bar{\B}_+,$}
\ee
and $|q(y)|_{C^2}\equiv 0$ for $y\in\dS^{n-1}\setminus\bar{\B}_+.$
It's easy to see that when $y\in\dS^{n-1}\setminus\bar{\B}_+$
\be\label{std2}
\td{\z}^k_{\bar{\B}_+}(x, y)\leq-M+\sqrt{l_k^2+(x_n+M)^2},
\ee
where the equality is achieved at $y=(0,\cdots, 0, 1).$
Therefore, we only need to show when $y_1>0$ the above inequality \eqref{std2} still holds.
By \eqref{std1} we know that this is equivalent to show
\[a_0y_1^2+\sqrt{l_k^2+|Dq|^2+M^2-p_1^2+x_n^2+2p_nx_n}-\sqrt{l_k^2+(x_n+M)^2}\leq 0.\]

By \eqref{std1} we have $p_1^2=(q_1+My_1)^2\geq(M-a_1)^2y_1^2.$ In the following, we need to estimate $|p_n|.$
Suppose $\{e_1,e_2,\cdots,e_{n-1}\}$ is an orthonormal frame on the unit sphere $\dS^{n-1}$. Suppose $\{x_1,x_2,\cdots,x_n \}$ is the coordinate of $\R^n$ and $r$ is the polar radius, then for  $l=1,\cdots, n$, we let
$$\frac{\p}{\p x_l}=\sum_{i=1}^{n-1}a_l^ie_i+b_l\frac{\p}{\p r},$$ where $\sum_i(a_l^i)^2+b_l^2=1$.
We denote the boundary of $\mathcal{B}_+$ to be $S$, which is an $(n-2)$-dimensional sphere. Note that $Dq=(\frac{\p q}{\p x_1},\frac{\p q}{\p x_2},\cdots, \frac{\p q}{\p x_n})$, we have
$$\left.\frac{\p q}{\p x_n}\right|_{S}=\sum\limits_{i=1}^{n-1}a_n^ie_i(q)|_S+b_n\frac{\p q}{\p r}=0,$$ and
$$\left.\frac{\p^2 q}{\p x_1\p x_n}\right|_S=\sum\limits_{i,j=1}^{n-1}a_1^je_j(a_n^ie_1(q))|_S+b_1|_S\left.\frac{\p \sum_i a_n^ie_i(q)}{\p r}\right|_S=0.$$ Here we have used $b_1|_S\equiv0$.  Thus, using L'Hospital rule, we get
\begin{eqnarray}
\lim_{y_1\goto 0+} \frac{q_n(y_1,\omega\sqrt{1-y_1^2})}{y_1^2}&=&\lim_{y_1\goto 0+}\frac{q_{n1}+\sum_{i=2}^nq_{ni}\omega_i\frac{-y_1}{\sqrt{1-y_1^2}}}{2y_1}\nonumber\\
&=&\lim_{y_1\goto 0+}\frac{q_{n1}}{2y_1}-\sum_{i=2}^nq_{ni}\omega_i=\frac{q_{n11}}{2}-\sum_{i=2}^nq_{ni}\omega_i,\nonumber
\end{eqnarray}
where $\omega=(\omega_2, \cdots, \omega_n)\in S.$
Therefore, there exist some constants $\delta, a_3>0$ such that when $y_1<\delta$
$$|q_n|\leq a_3 y_1^2.$$
On the other hand, it's easy to see that when $y_1\geq\delta$, we have $$|q_n|\leq |Dq|\leq a_1y_1\leq \frac{a_1}{\delta}y_1^2.$$
We conclude that when $y_1\in [0,1],$ there exists a constant $a_4>0$ such that $|q_n|\leq a_4y_1^2.$

This implies
\be\label{std3}
\begin{aligned}
&\sqrt{l_k^2+(x_n+M)^2}-\sqrt{l_k^2+|Dq|^2+M^2-p_1^2+x_n^2+2p_nx_n}\\
&=\frac{2x_nM-|Dq|^2+p_1^2-2p_nx_n}{\sqrt{l_k^2+(x_n+M)^2}+\sqrt{l_k^2+\sum\limits_{i=2}^{n-1}p_i^2+(x_n+p_n)^2}}\\
&\geq\frac{2x_nM-a_1^2y_1^2+(M-a_1)^2y_1^2-2a_4y_1\sqrt{1-y_n^2}x_n-2M|y_n|x_n}
{\sqrt{l_k^2+(x_n+M)^2}+\sqrt{l_k^2+\sum\limits_{i=2}^{n-1}p_i^2+(x_n+p_n)^2}}\\
&\geq\frac{2x_nM(1-|y_n|)\lt[1-\frac{a_4(1+|y_n|)}{M}\rt]+[(M-a_1)^2-a_1^2]y_1^2}
{\sqrt{l_k^2+(x_n+M)^2}+\sqrt{l_k^2+a_1^2+(x_n+M+a_1)^2}}\\
\end{aligned}
\ee
When $0\leq x_n\leq 10M$ and $M>2a_2$ very large we have
\[\frac{2x_nM(1-|y_n|)\lt[1-\frac{a_4(1+|y_n|)}{M}\rt]+[(M-a_1)^2-a_1^2]y_1^2}
{\sqrt{l_k^2+(x_n+M)^2}+\sqrt{l_k^2+a_1^2+(x_n+M+a_1)^2}}\geq \frac{[(M-a_1)^2-a_1^2]y_1^2}{24M}>a_0y^2_1.\]
When $x_n>10M$ we have
\begin{align*}
&\frac{2x_nM(1-|y_n|)\lt[1-\frac{a_4(1+|y_n|)}{M}\rt]+[(M-a_1)^2-a_1^2]y_1^2}
{\sqrt{l_k^2+(x_n+M)^2}+\sqrt{l_k^2+a_1^2+(x_n+M+a_1)^2}}\\
&\geq \frac{x_nM(1-|y_n|)}{3x_n}=\frac{M}{3}(1-|y_n|)>a_0(1-y_n^2).
\end{align*}
This completes the proof of this Lemma.
\end{proof}

From the above analysis we can see that as $r\goto\infty$
\be\label{sub-asymptotic}
\left\{
\begin{aligned}
&\lu^k_{\bar{\B}_+}(r\theta)-V_{\bar{\B}_+}(r\theta)\goto\sqrt{l_k^2+M^2}-M<l_k,\,\,\text{when {\bf $\theta\notin\B_+$}
is perpendicular to $\p_0 B_1^+,$ }\\
&\lu^k_{\bar{\B}_+}(r\theta)-V_{\bar{\B}_+}(r\theta)\goto q(\theta),\,\,
\text{otherwise.}\\
\end{aligned}
\right.
\ee

\subsection{Construction of supersolution}
\label{constructsup}
Let $\hp(y)=Dq(y)-My,$ where $M>2a_2$ is the same constant as in the construction of the subsolution.
Denote
\begin{align*}
\hz^k_{\bar{\B}_+}(x, y)&=q(y)+MV_{\bar{\B}_+}(y)+\z^k_{\bar{\B}_+}(x+\hp(y))\\
&=q(y)+MV_{\bar{\B}_+}(y)+\sqrt{f_k^2(x_1+\hp_1)+\sum\limits_{i=2}^n(x_i+\hp_i)^2}.\\
\end{align*}
\begin{lemm}
\label{addlem4.2}
 If as $|x|\goto\infty$, $V_{\bar{\B}_+}(x)\goto\infty,$ then we have
$\inf\limits_{y\in\dS^{n-1}}\tz^k_{\bar{\B}_+}(x, y)-V_{\bar{\B}_+}(x)\goto q\lt(\frac{x}{|x|}\rt).$
\end{lemm}
\begin{proof} We will divide this proof into 3 cases.

\textbf{Case 1.} When $\frac{x}{|x|}\in\B_+$ we assume
\[x=r\cos\beta e_1+r\sin\beta e_n,\,\,\cos\beta>0.\]
We also assume $\sin\beta\geq 0.$ Then as $r\goto\infty,$
when $y\in\bar{\B}_+,$ following the discussion in Subsection \ref{constructionsub} we get,
$\hz^k_{\bar{\B}_+}(x, y)-r\geq q\lt(\frac{x}{|x|}\rt)$
and the equality is achieved when $y=\frac{x}{|x|}.$
When $y\notin\bar{\B}_+$ we get
\begin{align*}
\hz^k_{\bar{\B}_+}(x, y)-r&\goto MV_{\bar{\B}_+}(y)+\sqrt{1+(r\cos\beta-My_1)^2+(r\sin\beta-My_n)^2+M^2(1-y_1^2-y_n^2)}-r\\
&\goto MV_{\bar{\B}_+}(y)-M\cos\beta y_1-M\sin\beta y_n.
\end{align*}
Recall that $\lt|q\lt(\frac{x}{|x|}\rt)\rt|\leq a_0\cos^2\beta,$ as $r\goto \infty$
we can see that
\begin{align*}
&\hz^k_{\bar{\B}_+}(x, y)-r-q\lt(\frac{x}{|x|}\rt)\\
&\geq M(1-\sin\beta)\sqrt{1-y_1^2}+M\cos\beta|y_1|-a_0(1-\sin^2\beta)\geq 0.
\end{align*}
Therefore, when $\frac{x}{|x|}\in\B_+$ and $|x|\goto\infty$ we have
\[\inf\limits_{y\in\dS^{n-1}}\hz^k_{\bar{\B}_+}-V_{\bar{\B}_+}(x)\goto q\lt(\frac{x}{|x|}\rt).\]

\textbf{Case 2.} When $\frac{x}{|x|}\in\p\B_+,$ without loss of generality, we assume
$x=(0, \cdots, 0, r)=re_n.$ As $r\goto 0,$ when $y\in\bar{\B}_+,$ it's easy to see that
\[\hz^k_{\bar{\B}_+}(x, y)-r\geq 0.\]
When $y\notin \bar{\B}_+$ we have
\begin{align*}
\hz^k_{\bar{\B}_+}(x, y)-r&\goto MV_{\bar{\B}_+}(y)+\sqrt{f_k^2(-My_1)+(r-My_n)^2+M^2(1-y_1^2-y_n^2)}-r\\
&\goto MV_{\bar{\B}_+}(y)-My_n\geq 0.
\end{align*}
Here the equality is achieved when $y=(y_1, 0, \cdots, 0, y_n)$ with $y_n\geq 0.$

\textbf{Case 3.} When $\frac{x}{|x|}\notin\bar{\B}_+$ and $V_{\bar{\B}_+}\lt(\frac{x}{|x|}\rt)\neq 0.$
Without loss of generality we assume $x=(r\cos\beta, 0, \cdots, 0, r\sin\beta),$ where
$\cos\beta<0$ and $\sin\beta>0.$ Then as $r\goto\infty$ we have when $y\in\bar{\B}_+$
\begin{align*}
&\hz^k_{\bar{\B}_+}(x, y)-r\sin\beta\\
&\goto q(y)+M+\sqrt{l_k^2+(r\sin\beta+\hat{p}_n(y))^2+\sum\limits_{i=2}^{n-1}\hat{p}_i^2}-r\sin\beta\\
&=q(y)+M+\lt<e_n, \hat{p}(y)\rt>\geq 0,
\end{align*}
where $0$ is achieved when $y=e_n.$
When $y\notin\bar{\B}_+$
\begin{align*}
&\hz^k_{\bar{\B}_+}(x, y)-r\sin\beta\\
&\goto MV_{\bar{\B}_+}(y)-My_n\geq 0,
\end{align*}
and $0$ is achieved when $y=(y_1, 0, \cdots, 0, y_n)$ with $y_n\geq 0.$

Note that when $V_{\bar{\B}_+}\lt(\frac{x}{|x|}\rt)\goto0$ and $V_{\bar{\B}_+}(x)\goto\infty$ as $|x|\goto\infty$ follow the same discussion as in case 3, we can show
that $\inf\limits_{y\in\dS^{n-1}}\tz^k_{\bar{\B}_+}(x, y)-V_{\bar{\B}_+}(x)\goto 0.$ This completes the proof of Lemma \ref{addlem4.2}.
\end{proof}

Now consider the case when $V_{\bar{\B}_+}\lt(\frac{x}{|x|}\rt)\goto0$ and $V_{\bar{\B}_+}(x)<\infty$ as $|x|\goto\infty.$ Without loss of generality we assume
$x=(x_1, 0, \cdots, 0, x_n)$ where $x_1\goto-\infty$ and $x_n\geq 0$ is bounded.
When $y\in\B_{\delta}\lt(\frac{x}{|x|}\rt),$ $0<\delta<\frac{\pi}{2}$ we have
\[\hz^k_{\bar{B}_+}(x, y)\goto MV_{\bar{\B}_+}(y)+\sqrt{l_k^2+(x_n-My_n)^2+M^2(1-y_1^2-y_n^2)}\]
Our supersolution is defined by
\[ \uu^k_{\bar{\B}_+}(x)=\inf\limits_{y\in\B_{\delta}\lt(\frac{x}{|x|}\rt)}\hz^k_{\bar{\B}_+}(x, y).\]
From above discussions it's clear that
when $V_{\bar{\B}_+}(x)\goto\infty$ as $|x|\goto\infty,$ we have $\uu^k_{\bar{\B}_+}(x)-V_{\bar{\B}_+}(x)\goto q\lt(\frac{x}{|x|}\rt).$  We want to show
$\uu^k_{\bar{\B}_+}(x)\geq \lu^k_{\bar{\B}_+}(x)$ as $|x|\goto\infty$ in all directions. Thus we need to prove the following Lemma.

\begin{lemm}
\label{subsuper-comparison lem}For any $\delta\in (0, \frac{\pi}{2}),$
when $y\in\B_{\delta}(-e_1)$ and $x_n\geq 0,$ we have
$$MV_{\bar{\B}_+}(y)+\sqrt{l_k^2+(x_n-My_n)^2+M^2(1-y_1^2-y_n^2)}\geq -M+\sqrt{l_k^2+(x_n+M)^2}.$$
\end{lemm}
\begin{proof}
It's easy to see that
\[L.H.S.\geq M|y_n|+\sqrt{l_k^2+(x_n-My_n)^2}.\]
If we can show $ M|y_n|+\sqrt{l_k^2+(x_n-My_n)^2}\geq  -M+\sqrt{l_k^2+(x_n+M)^2},$ then we would be done.
Notice that
\begin{align*}
&\sqrt{l_k^2+(x_n+M)^2}-\sqrt{l_k^2+(x_n-My_n)^2}\\
&=\frac{2x_nM(1+y_n)+M^2(1-y_n^2)}{\sqrt{l_k^2+(x_n+M)^2}+\sqrt{l_k^2+(x_n-My_n)^2}}\\
&\leq M(1+y_n)\leq M(1+|y_n|).
\end{align*}
Therefore the Lemma is proved.
\end{proof}

Since when $\theta\notin\B_+$ is perpendicular to $\p_0 B_1^+,$ by the above discussion we have, as $r\goto\infty$
\begin{align*}
\hz^k_{\bar{B}_+}(x, y)&\goto MV_{\bar{\B}_+}(y)+\sqrt{l_k^2+(My_n)^2+M^2(1-y_1^2-y_n^2)}\\
&=M\sqrt{1-y_1^2}+\sqrt{l_k^2+M^2(1-y_1^2)}\geq l_k.
\end{align*}
We conclude that as $|r|\goto\infty$
\be\label{super-asymptotic}
\left\{
\begin{aligned}
&\uu^k_{\bar{\B}_+}(r\theta)-V_{\bar{\B}_+}(r\theta)\goto l_k,\,\,\text{when {\bf $\theta\notin\B_+$}
is perpendicular to $\p_0 B_1^+$ }\\
&\uu^k_{\bar{\B}_+}(r\theta)-V_{\bar{\B}_+}(r\theta)\goto q(\theta),\,\,
\text{otherwise.}\\
\end{aligned}
\right.
\ee

\subsection{Construction of spacelike cutoff function}
\label{constructcut}
Recall in \cite{WX} we proved that the following function is spacelike.
\begin{lemm}
\label{lem23}(See Lemma 23 in \cite{WX} )
Let $A_0=A_0(\lambda),$ $B_0=B_0(\lambda)$ be large numbers depending on $\lambda\in(0, 1].$ Then when $R_0>A_0,$ $R_1>B_0R_0,$
\be\label{cutoff-function}
\psi_{\bar{\B}_+}=\left\{
\begin{aligned}
&\sqrt{\lambda^2+V^2_{\bar{\B}_+}(x)}+\frac{1}{\sqrt{1+|\bar{x}|^2}}\lt(1-V_{\bar{\B}_+}\lt(\frac{x}{|x|}\rt)\rt),\,\, |x|\geq R_1\\
&\sqrt{\lambda^2+V^2_{\bar{\B}_+}(x)}+\frac{1}{\sqrt{1+|\bar{x}|^2}}\lt(1-V_{\bar{\B}_+}\lt(\frac{x}{|x|}\rt)\rt)\eta(x),\,\, R_0<|x|<R_1\\
&\sqrt{\lambda^2+V^2_{\bar{\B}_+}(x)},\,\,|x|\leq R_0
\end{aligned}
\right.
\ee
is spacelike on $\R^n,$ where $\eta(x)=\frac{|x|-R_0}{R_1-R_0}.$
\end{lemm}
Now let $\pM_{\bar{\B}_+}(x)=M_1\psi_{\bar{\B}_+}\lt(\frac{x}{M_1}\rt),$ where $M_1>0$ is a large constant to be determined. Then $\pM_{\bar{\B}_+}(x)$ is sapcelike on $\R^n$ as well. Moreover, similar to the behavior of $\z^k_{\bar{\B}_+},$ at infinity,
for $\theta=(\theta_1,\theta_2,\cdots, \theta_n)\in\mathbb{S}^{n-1}$,
we have
\[\lim\limits_{r\goto\infty}{\left(\pM_{\bar{\B}_+}(r\theta)-V_{\bar{\B}_+}( r\theta)\right)}=
\left\{
\begin{aligned}
&M_1(\lambda+1),\,\,\theta\bot\T_0\bar{B}_1^+:=\{\xi\in\bar{B}_1, \xi_1=0\}\text{ and } \theta_1=-1\\
&0,\,\,\text{elsewhere.}
\end{aligned}
\right.
\]
We will consider
\[\pM_{\bar{\B}_+}(x, y)=q(y)-M+\pM_{\bar{\B}_+}(x+p(y)),\]
where $p(y)=Dq(y)+My.$ Note that when $|x|\geq M_1R_1$
\[\pM_{\bar{\B}_+}(x)=\sqrt{M_1^2\lambda^2+V^2_{\bar{\B}_+}(x)}+\frac{M_1^2}{\sqrt{M_1^2+|\bar{x}|^2}}\lt(1-V_{\bar{\B}_+}\lt(\frac{x}{|x|}\rt)\rt).\]
When $V_{\bar{\B}_+}(x)\goto\infty$ as $|x|\goto\infty,$ we can discuss like in the Subsection \ref{constructionsub} and obtain
\[\sup\limits_{y\in\dS^{n-1}}\pM_{\bar{\B}_+}(x, y)-V_{\bar{\B}_+}(x)\goto q\lt(\frac{x}{|x|}\rt).\]
We only need to look at the case when $V_{\bar{\B}_+}\lt(\frac{x}{|x|}\rt)\goto 0$ and $V_{\bar{\B}_+}(x)$ is bounded.
Without loss of generality, let's assume $x=(x_1, 0, \cdots, 0, x_n)$ where $x_1\goto-\infty$ and $x_n\geq 0$
is bounded. It is easy to see that when $r\goto\infty$
\begin{align*}
&\pM_{\bar{\B}_+}(x, y)-V_{\bar{\B}_+}(x)\\
&\geq -M+V_{\bar{\B}_+}(x+p(y))+\frac{M_1^2}{\sqrt{M_1^2+|\bx+\bar{p}|^2}}-x_n.
\end{align*}
Therefore, let $y=(0, \cdots, 0, 1)$ we get
\be\label{cut1.1}
\begin{aligned}
&\sup\limits_{y\in\dS^{n-1}}\pM_{\bar{\B}_+}(x, y)-V_{\bar{\B}_+}(x)\\
&\geq\frac{M_1^2}{\sqrt{M_1^2+(x_n+M)^2}}.
\end{aligned}
\ee
In view of the definition of $\uu^k_{\bar{\B}_+}$ we know, when $x=(x_1, 0, \cdots, 0, x_n)$ where $x_1\goto-\infty$ and $x_n\geq 0$
is bounded, $\uu^k_{\bar{\B}_+}(x)-V_{\bar{\B}_+}(x)\leq \frac{l_k}{\sqrt{l_k^2+x_n^2}+x_n}.$

Combining the analysis above, it's straightforward to see that when $M_1=M_1(M, l_k)>0$ large, we have as $|x|\goto\infty,$
\[\td{\Psi}_{\bar{\B}_+}(x)-\uu^k_{\bar{\B}_+}(x)\geq 0,\]
here $\td{\Psi}_{\bar{\B}_+}(x):=\sup\limits_{y\in\dS^{n-1}}\pM_{\bar{\B}_+}(x, y);$ while in a compact set
$K\subset B_{R_0}(x),$ we can choose $\lambda>0$ small such that
$\lu^k_{\bar{\B}_+}(x)-\td{\Psi}_{\bar{\B}_+}(x)>0.$

\subsection{Construction for general $\F$}
\label{constructF}
In this subsection, we will construct the subsolution, supersolution, and the spacelike cutoff function for general prescribed
lightlike directions $\F\in\dS^{n-1}.$
Let's recall the following Lemma from \cite{BS}:
\begin{lemm}
\label{cblem2}(Lemma 4.5 in \cite{BS})
Let $\F$ be the closure of some open nonempty subset of the ideal boundary $\dS^{n-1}$ with $\T \F\in C^{1, 1}.$
There exists $\delta_0>0$ such that the following holds:

(1) $\F$ and $\bar{\F}^c$ are the union of closed balls of $\dS^{n-1}$ of radius $\delta_0.$

(2) For every $x\in\bar{\F}^c,$ there exists a closed ball $\B$ with radius bounded below by $\delta_0$ which contains $x$
and is contained in $\bar{\F}^c$ such that $d_S(x, \B^c)=d_S(x, \F).$
\end{lemm}

Now, for a given $\F,$ we fix $\delta_0>0$ as in Lemma \ref{cblem2}.
Let
\be\label{def-lz}
\ubar{\z}^k_\F(x)=\sup\limits_{\bar{\B}\subset \F, \delta(\bar{\B})\geq\delta_0}\z^k_{\bar{\B}}(x),\ee
where $\z^k_{\bar{\B}}(x)$
are semitroughs satisfying $\s_k(\ka[\z_{\bar{\B}}(x)])=\binom{n}{k}.$
Following the proof of Theorem 4.3 in \cite{BS}, we can show
\begin{lemm}
\label{cblem0}
Let $\mathbb{S}^{n-1}$ denote the ideal boundary at infinity of the hyperbolic space and $\F\subset\mathbb{S}^{n-1}$
be a non-empty closed subset. Assume that $\F$ is of the following form: if $n\geq 3,$ $\F$ is the closure of some open subset of
$\mathbb{S}^{n-1}$ with $C^{1, 1}$ boundary; if $n=2,$ $\F$ is a finite union of non-trivial intervals on the unit circle. Define
$V_\F:\R^n\goto\R$ by
\[V_\F(x):=\sup\limits_{\lambda\in\F}x\cdot\lambda.\]
Then $\ubar{\z}^k_\F,$ where $1\leq k\leq n,$ defined in \eqref{def-lz} satisfies the following properties:\\
(1) \be\label{cb0.1}
C>\ubar{\z}^k_\F-V_\F(x)>0 \,\,\mbox{for any $x\in\R^n,$}
\ee
where $C>0$ is a constant depending on $n, k,$ and $\delta_0.$ Moreover, if $\frac{x}{|x|}\in\F,$ then
\be\label{cb0.2}
\ubar{\z}^k_\F(x)-V_\F(x)\goto 0\,\,\mbox{as $|x|\goto\infty.$}
\ee
(2) For every compact subset $K\subset\R^n,$ there exists a constant $\eta>0$ such that for every $x\in K$
\be\label{cb0.3}
\ubar{\z}^k_\F(x)\geq V_\F(x)+\eta.
\ee
\end{lemm}

In the following, we assume the boundary function $q(y)$ satisfies
\be\label{q.1'}
q\in C^{2, 1}_{\text{loc}}\lt(\R^n\setminus\{0\}\rt)\,\,\mbox{and $q(x)=q\lt(\frac{x}{|x|}\rt)$},
\ee
and
\be\label{q.2'}
q(y)\equiv 0\,\,\mbox{for $y\in\dS^{n-1}\setminus\F.$}
\ee

First, let $\td{\z}^k_{\F}(x, y)=q(y)-M+\ubar{\z}^k_{\F}(x+p(y)),$
where $p(y)=Dq(y)+My.$ Then the subsolution is chosen as
\be\label{lower barrier}
\lu^k_{\F}(x)=\sup\limits_{y\in\dS^{n-1}}\td{\z}^k_{\F}(x, y).
\ee

Next, let $\bar{\z}^k_{\F}(x)=\inf\limits_{\bar{\B}\supset\F, \delta(\bar{\B})\leq\pi-\delta_0}\z^k_{\bar{\B}}(x)$
and $\hz^k_{\F}(x, y)=q(y)+MV_\F(y)+\bar{\z}^k_{\F}(x+\hp(y)),$
where $\hp(y)=Dq(y)-My.$ Then the supersolution is constructed as
\[\uu^k_{\F}(x)=\inf\limits_{y\in\bar{\B}_{\delta}\lt(\frac{x}{|x|}\rt)}\hz^k_{\F}(x, y).\]
We choose $\delta>0$ so small such that for any $\frac{x}{|x|}\perp(\p\td{F}\cap B_1)$ and $\frac{x}{|x|}\notin\F,$
$\bar{\B}_{\delta}\lt(\frac{x}{|x|}\rt)\cap\F=\emptyset.$ Here $\td{F}$ is the convex hull of $\F$ in $B_1.$

Finally, let $\pM_{\F}(x)=\inf\limits_{\bar{\B}\supset\F, \delta(\bar{\B})\leq\pi-\delta_0}\psi^{M_1}_{\bar{\B}}(x),$
where $\pM_{\bar{\B}}(x)$ is the Lorentz transform of $\pM_{\bar{\B}_+}(x).$
Let $\pM_{\F}(x, y)=q(y)-M+\pM_{\F}(x+p(y)),$
where $p(y)=Dq(y)+My.$ Then the spacelike cutoff function is chosen as
\be\label{cutoff-function-F}
\td{\Psi}_{\F}(x)=\sup\limits_{y\in\dS^{n-1}}\pM_{\F}(x, y).
\ee

In the following, we will explain why the subsolution, supersolution, and spacelike cutoff function satisfy the condition:
as $|x|\goto \infty,$ $\td{\Psi}_{\F}(x)\geq \uu^k_\F\geq\lu^k_\F(x).$
In view of \eqref{lower barrier} we have
\be\label{exp.1}
\begin{aligned}
\lu^k_\F(x)&=\sup\limits_{y\in\dS^{n-1}}\left\{ q(y)-M+\ubar{\z}^k_\F(x+p(y))\right\}\\
&=\sup\limits_{y\in\dS^{n-1}}\lt\{q(y)-M+\sup\limits_{\bar{\B}\subset\F, \delta(\bar{\B})\geq\delta_0}\z^k_{\bar{\B}}(x+p(y))\rt\}\\
&=\sup\limits_{y\in\dS^{n-1}}\lt\{q(y)-M+\z^k_{\bar{\B}_{x, y}}(x+p(y))\rt\}\\
&=q(y_0)-M+\z^k_{\bar{\B}_0}(x+p(y_0)).
\end{aligned}
\ee
We want to point out that since the choice of $y_0$ depends on $x,$ the choice of $\bar{\B}_0$ only depends on $x$ as well.
From \eqref{exp.1} we can see that in order to understand the asymptotic behavior of $\lu^k_\F$ as $|x|\goto\infty,$ we only need to
understand the behavior of
\[\z^k_{\bar{\B}_0}(x)-V_{\bar{\B}_0}(x)\,\,\mbox{as $|x|\goto\infty.$}\]
After a rotation, we may assume
\[(-1, 0, \cdots, 0)\perp\p\td{B}_0\cap B_1,\]
where $\td{B}_0$ is the convex hull of $\bar{\B}_0$ in $B_1.$ Then by \eqref{cb1.2} we get
\be\label{exp.2}
\z^k_{\bar{\B}_0}(\td{x}_1', \bar{x})-V_{\bar{\B}_0}(x_1', \bar{x})=
\frac{\z^k_{\bar{\B}_+}(x)-V_{\bar{\B}_+}(x)}{\sqrt{1-\al^2}},
\ee
where $\td{x}_1'=\frac{x_1-\al\z^k_{\bar{\B}_+}(x)}{\sqrt{1-\alpha^2}},$
$x_1'=\frac{x_1-\al V_{\bar{\B}_+}(x)}{\sqrt{1-\al^2}},$ and $\bar{x}=(x_2, \cdots, x_n).$
Therefore
\be\label{exp.3}
\begin{aligned}
&\z^k_{\bar{\B}_0}(\td{x}_1', \bar{x})-V_{\bar{\B}_0}(\td{x}_1', \bar{x})\\
&=\z^k_{\bar{\B}_0}(\td{x}_1', \bar{x})-V_{\bar{\B}_0}(x_1', \bar{x})+V_{\bar{\B}_0}(x_1', \bar{x})-V_{\bar{\B}_0}(\td{x}_1', \bar{x})\\
&=\frac{\z^k_{\bar{\B}_+}(x)-V_{\bar{\B}_+}(x)}{\sqrt{1-\al^2}}+V_{\bar{\B}_0}(x_1', \bar{x})-V_{\bar{\B}_0}(\td{x}_1', \bar{x}).
\end{aligned}
\ee
Notice that as $|x|\goto\infty,$ when $\frac{x}{|x|}\nrightarrow (-1, 0, \cdots, 0)$
we have $\td{x}_1\goto x_1',$ which yields
\[\z^k_{\bar{\B}_0}(\td{x}_1', \bar{x})-V_{\bar{\B}_0}(\td{x}_1', \bar{x})\goto 0.\]
When $\frac{x}{|x|}\goto (-1, 0, \cdots, 0),$ a direct calculation gives
\[\z^k_{\bar{\B}_0}(\td{x}_1', \bar{x})-V_{\bar{\B}_0}(\td{x}_1', \bar{x})
\goto \sqrt{1-\al^2}\lt(\sqrt{l_k^2+|\bar{x}|^2}-|\bar{x}|\rt).\]
We conclude
\be\label{exp.4}
\z^k_{\bar{\B}_0}(x)-V_{\bar{\B}_0}(x)\goto
\lt\{
\begin{aligned}
&\sqrt{1-\al^2}\lt(\sqrt{l_k^2+|\bar{x}|^2}-|\bar{x}|\rt),\,\,\mbox{when $\frac{x}{|x|}\goto(-1, 0, \cdots, 0)$}\\
&0,\,\,\mbox{otherwise.}
\end{aligned}
\rt.
\ee

Combining \eqref{exp.4} with \eqref{exp.1}, following the same argument as in Subsection \ref{constructionsub}, we obtain that
as $|x|\goto \infty,$ when $\frac{x}{|x|}\goto \F,$ $\lu^k_\F(x)-V_\F(x)\goto q\lt(\frac{x}{|x|}\rt).$ Similarly, we can analyze the asymptotic behaviors
of $\uu^k_\F(x)$ and $\td{\Psi}_\F(x).$ We verify that they satisfy for any compact set
$K\subset\R^n,$ there exists a $\td{\Psi}_\F(x)$ such that when $x\in K,$ $\td{\Psi}_\F(x)<\lu^k_\F(x)$ and as $|x|\goto\infty,$ $\td{\Psi}_\F(x)\geq\uu^k_\F(x).$ Moreover, for any $x\in\R^n,$ $\uu^k_\F(x)>\lu^k_\F(x).$

\subsection{Upper barrier, lower barrier, and spacelike cutoff function}
\label{uls}
We start with constructing the upper barrier. Consider $\lu_\F^1$ and $\uu^1_\F$ constructed in the
Subsection \ref{constructF}. It's easy to see that $\lu^1_\F(x)$ and $\uu^1_\F(x)$
are weak sub and super solutions to the prescribed mean curvature equation
\be
\left\{
\begin{aligned}
&\textup{div}\lt(\frac{Du}{\sqrt{1-|Du|^2}}\rt)=n\\
&|Du(x)|<1\,\,\text{for all}\,\,x\in\R^n.
\end{aligned}
\right.
\ee
Moreover, as $|x|\goto\infty$ for $\frac{x}{|x|}\in\F,$ we have
\be\label{barrier-asymptotic}
\uu^1_\F(x),\lu^1_\F(x)\goto q\lt(\frac{x}{|x|}\rt).
\ee
By Theorem 6.1 of \cite{CT} we know that there exists a strictly convex function $h(x)$ satisfying
\[\lu^1_\F(x)\leq h(x)\leq\uu^1_\F(x)\,\,\text{for all}\,\, x\in\R^n\]
and $\s_1(\ka[\M_h(x)])=n.$ In the rest of this paper, we will use this $h(x)$ as our upper barrier.
We will denote our lower barrier by $\lu(x),$ which is given by \eqref{lower barrier}. Note that our upper barrier
$h(x)$ is smooth while our lower barrier is only Lipschitz. In view of the discussions in Section \ref{properties of semitroughs},
Subsection \ref{constructcut}, and Subsection \ref{constructF}, we obtain the
following Lemma.
\begin{lemm}
\label{cutoff lemma}
For every compact subset $K\subset\R^n,$ there exists $\lambda\in(0, 1)$ small and $M_1>0$ large such that
the spacelike function $\td{\psi}_\F(x)$ constructed in \eqref{cutoff-function-F} satisfies
\be\label{cutoff1}
\td{\psi}_\F(x)-\lu(x)>c_0>0\,\,\text{in $K$}
\ee
and
\be\label{cutoff2}
\td{\psi}_\F(x)-h(x)\geq0\,\,\text{as $|x|\goto\infty$.}
\ee
\end{lemm}
This Lemma will be used to obtain the local $C^1$ estimates.

\section{Legendre transform of barrier functions}
\label{ltbf}
For any given $\mathcal{F}\subset\mathbb{S}^{n-1}$,
since we can view every ball $\mathcal{B}$ satisfying $\bar{\mathcal{B}}\subset \mathcal{F}$ and $\delta_0\leq\delta(\bar{\mathcal{B}})\leq \frac{\pi}{2}$, as a point $(x,r)\in\mathbb{R}^{n+1}$, where $x,r$ are the center and radius of $\bar{\mathcal{B}}$ respectively. It is clear that the set of all $(x,r)$ is a compact set with respect to the standard metric of $\mathbb{R}^{n+1}$. Therefore, using  the standard compactness argument, we have for any given $x$, there exists some ball $\mathcal{B}$ such that
\be\label{ltbf0}
\ubar{\z}^k_\F(x)=\z^k_{\bar{\mathcal{B}}}(x).
\ee

\begin{lemm}
\label{ltbflem1}
Let $\lu(x)$ be the subsolution constructed in Subsection \ref{constructF}. Then $\lu(x)$ is a strictly convex and spacelike function over $\R^n.$
\end{lemm}
\begin{proof}
For any $x_1\neq x_2$, in view of \eqref{lower barrier} and \eqref{ltbf0}, we have
\be\label{ltbf1}
\begin{aligned}
\lu\lt(\frac{x_1+x_2}{2}\rt)&=\sup\limits_{y\in\dS^{n-1}}\tz^k_{\F}\lt(\frac{x_1+x_2}{2}, y\rt)\\
&=\tz^k_{\F}\lt(\frac{x_1+x_2}{2}, y_0\rt)=q(y_0)-M+\z^k_{\bar{\B_0}}\lt(\frac{x_1+x_2}{2}+p(y_0)\rt).\\
\end{aligned}
\ee
By the strict convexity of $\z^k_{\bar{\B_0}},$ we get
\begin{align*}
\lu\lt(\frac{x_1+x_2}{2}\rt)&<q(y_0)-M+\frac{1}{2}\z^k_{\bar{\B_0}}\lt(x_1+p(y_0)\rt)
+\frac{1}{2}\z^k_{\bar{\B_0}}\lt(x_2+p(y_0)\rt)\\
&\leq\frac{\lu(x_1)+\lu(x_2)}{2}.
\end{align*}
This proves that $\lu(x)$ is strictly convex.

Next, for any $x, y\in\R^n,$ since $\lu(x)=q(y_x)-M+\z^k_{\bar{\B_x}}\lt(x+p(y_x)\rt),$
$\lu(y)\geq q(y_x)-M+\z^k_{\bar{\B_x}}\lt(y+p(y_x)\rt),$ and $\z^k_{\bar{\B_x}}$ is spacelike we obtain
\[\lu(x)-\lu(y)<|x-y|.\]
Similarly, we derive
\[\lu(y)-\lu(x)<|x-y|.\]
This proves that $\lu(x)$ is spacelike.
\end{proof}

Therefore, we can define the Legendre transform of
$\lu$ as follows
\[\lus(\xi):= \sup\limits_{x\in\R^n}\lt(\xi\cdot x-\lu(x)\rt).\]
\begin{lemm}
\label{cblem-domain}
The domain of $\lus(\xi),$ which we will denote by $D^*,$ is
\[\hat{F}=\text{Conv}(\F)\subseteq\bar{B}_1=\lt\{\xi\in\R^n\big||\xi|\leq 1\rt\}.\]
\end{lemm}
\begin{proof}
By the definition of Legendre transform we know
\[D^*:=\{\xi\in\bar{B}_1| \sup\limits_{x\in\R^n}(\xi\cdot x-\lu(x))<\infty\}.\]
Let's denote
\[V_{\lu}(x)=\lim\limits_{r\goto\infty}\frac{\lu(rx)}{r}.\]
In view of Lemma \ref{cblem0}, we have $V_{\lu}=V_{\F}$.
Then following the proof of Lemma 4.6 of \cite{CT} we have
\[\Chi_{V_{\lu}}(0):=\{p\in\R^n| V_{\lu}(x)\geq p\cdot x+V_{\lu}(0)\,\,\mbox{for all $x\in\R^n$}\}=\hat{F}.\]
Next, we want to show $D^*=\Chi_{V_{\lu}}(0).$

\textbf{Step 1.} In this step, we will show $\Chi_{V_{\lu}}(0)\subseteq D^*.$ For any $\xi_0\in \Chi_{V_{\lu}}(0),$ by the definition of
$\Chi_{V_{\lu}}(0)$ we have
\be\label{cb0.4}
V_{\lu}(x)\geq \xi_0\cdot x\,\,\mbox{for all $x\in\R^n.$}
\ee
From Lemma 4.3 of \cite{CT} we can see
\be\label{cb0.5}
V_{\lu}(x)=\sup\limits_{\xi\in\F}x\cdot\xi=V_{\F}(x).
\ee
Moreover, by Lemma \ref{cblem0} we know that there exists $C>0$ such that
\be\label{cb0.6}
|\lu(x)-V_{\F}(x)|\leq C\,\,\mbox{for all $x\in\R^n$.}
\ee
Combining equations \eqref{cb0.4}-\eqref{cb0.6} we conclude that
\[\xi_0\cdot x-\lu(x)\leq\xi_0\cdot x-V_{\F}(x)+C\leq C\,\,\mbox{for all $x\in\R^n$,}\]
thus Step 1 is proved.

\textbf{Step 2.} We will show $D^*\subseteq \Chi_{V_{\lu}}(0).$ This is equivalent to show that
if $\xi\in \Chi_{V_{\lu}}(0)^c,$ then $\xi\notin D^*.$
Now, let's choose any $\xi_1\in\Chi_{V_{\lu}}(0)^c.$ Then, by the definition of $\Chi_{V_{\lu}}(0)$ we know there exists
$x_1\in\R^n\setminus\{0\}$ such that
\[\xi_1\cdot x_1-V_{\lu}(x_1)=c_0>0.\]
Since both $V_{\lu}(x)$ and $\xi_1\cdot x$ are homogenous of degree one, we have
\[r\xi_1\cdot x_1-V_{\lu}(rx_1)=rc_0\goto\infty\,\,\mbox{as $r\goto\infty$}.\]
In view of \eqref{cb0.6} we know $\xi_1\notin D^*.$ This completes the proof of Step 2 thus the Lemma.
\end{proof}

\begin{lemm}
\label{cblem5}
For any $\xi\in\text{Int}D^*,$ there exists a unique $x_0\in\R^n$ such that
\[\lus(\xi)=x_0\cdot\xi-\lu(x_0).\]
\end{lemm}
\begin{proof} \textbf{Step 1.} In this step we will show that for any $\xi\in\text{Int}D^*,$
$\sup\limits_{x\in \R^n}\{x\cdot\xi-\lu(x)\}$ is achieved away from $\infty.$
We will prove it by contradiction. If not, then there exists a sequence $\{x_i\}$
with $|x_i|\goto\infty$ as $i\goto\infty,$ such that
\[|x_i\cdot\xi-\lu(x_i)-\lus(\xi)|<\frac{1}{i}.\]
Denoting $\theta_i:=\frac{x_i}{|x_i|},$ then we have
\[\frac{\lus(\xi)-\frac{1}{i}}{|x_i|}<\theta_i\cdot\xi-\frac{\lu(|x_i|\theta_i)}{|x_i|}<\frac{\lus(\xi)+\frac{1}{i}}{|x_i|}.\]
Choosing a convergent subsequence of $\{\theta_i\},$ which we still denote by $\{\theta_i\},$ and let $i\goto\infty,$ we get
\[\theta\cdot\xi-V_{\lu}(\theta)=0.\]
This contradicts the assumption that $\xi\in\text{Int}D^*.$ Hence, Step 1 is proved.

\textbf{Step 2.} We will show that $x_0$ is unique.
If not, let's assume for some $\xi\in\text{Int}D^*,$ there exist $x_1, x_2\in\R^n$ such that
\[\lus(\xi)=\xi\cdot x_1-\lu(x_1)=\xi\cdot x_2-\lu(x_2).\]
Since $\lu$ is strictly convex, we have
\[\frac{\lu(x_1)+\lu(x_2)}{2}>\lu\lt(\frac{x_1+x_2}{2}\rt).\]
By a straightforward calculation we obtain,
\begin{align*}
&\xi\cdot\lt(\frac{x_1+x_2}{2}\rt)-\lu\lt(\frac{x_1+x_2}{2}\rt)\\
&>\xi\cdot\lt(\frac{x_1+x_2}{2}\rt)-\frac{\lu(x_1)+\lu(x_2)}{2}\\
&=\xi\cdot\lt(\frac{x_1+x_2}{2}\rt)-\frac{2\lu(x_1)+\xi\cdot(x_2-x_1)}{2}\\
&=\xi\cdot x_1-\lu(x_1)=\lus(\xi).
\end{align*}
This leads to a contradiction.
\end{proof}

From now on, we will say $D\lu^*(\xi)=x$ if
\be\label{def-gradient}
\lu^*(\xi)=x\cdot\xi-\lu(x).
\ee
We will show that the map $D\lu^*(\cdot): \text{Int}D^*\mapsto\R^n$ is continuous in the next Lemma.
\begin{lemm}
\label{cblem6}
$\lu^*$ is a convex $C^1$ function on $\text{Int}D^*$.
\end{lemm}
\begin{proof} The definition of $\lu^*$ says it is the supreme of a family of linear functions, which implies $\lu^*\in C^{0,1}$. Moreover, in view of Lemma \ref{cblem5}, using a similar argument as the proof of the strict convexity of $\lu$, we obtain the convexity of $\lu^*$. In the following,
we will show $\lu^*\in C^1$.

Suppose $D\lu^*(\xi_0)=x_0$. The set of sub differential of $\lu^*$ is  defined by
$$\partial{\lu^*}(\xi_0)=\{\eta\in\mathbb{R}^n; \lu^*(\xi)-\lu^*(\xi_0)\geq \eta\cdot(\xi-\xi_0), \text{ for any } \xi\in D^*\}.$$ By the definition of Legendre transform, we have
$$\lu^*(\xi)\geq \xi\cdot x_0-\lu(x_0),$$ which implies
$$\lu^*(\xi)-\lu^*(\xi_0)\geq \xi\cdot x_0-\lu(x_0)-(\xi_0\cdot x_0-\lu(x_0))=x_0\cdot(\xi-\xi_0).$$ Therefore, we get $x_0\in\partial\lu^*(\xi_0)$.

Next, we will show $x_0$ is the unique element of $\partial\lu^*(\xi_0)$. Suppose $y_0\neq x_0$ is another element in $\p\lu^*(\xi_0)$. Since the sub differential set of $\lu$ is always nonempty.
We let $\tilde{\xi}\in\partial\lu(y_0)$, then for any $x\neq y_0$, we have
$$\lu(x)-\lu(y_0)\geq\tilde{\xi}\cdot(x-y_0),$$ which can be rewritten as
$$\tilde{\xi}\cdot y_0-\lu(y_0)\geq\tilde{\xi}\cdot x-\lu(x)\,\, \mbox{for $\forall x\in\R^n$}.$$
Therefore, we have $$\lu^*(\tilde{\xi})=\tilde{\xi}\cdot y_0-\lu(y_0).$$
Using $y_0\in\partial\lu^*(\xi_0)$, we have
$$\lu^*(\tilde{\xi})-\lu^*(\xi_0)\geq y_0\cdot(\tilde{\xi}-\xi_0),$$ which gives
$$y_0\cdot\xi_0-\lu(y_0)\geq \lu^*(\xi_0).$$ This contradicts Lemma \ref{cblem5}.
Thus, the notation $D\lu^*$ is the derivation of $\lu^*$.

Finally, let's show $D\lu^*$ is continuous on $\text{Int} D^*$. We will prove it by contradiction. If not, then there exists a sequence $\{\td{\xi}_i\}\goto \xi_0$ as $i\goto\infty,$
such that $D\lus(\td{\xi}_i)=\td{x}_i,$ $D\lus(\xi_0)=x_0,$ and $|\td{x}_i-x_0|>\delta_0>0$ for any $i\in\mathbb{N}.$
By the definition of Legendre transform we have
\[\td{x}_i\cdot\td{\xi}_i-\lu(\td{x}_i)>\td{\xi}_i\cdot x_0-\lu(x_0)\]
and
\[x_0\cdot\xi_0-\lu(x_0)>\xi_0\cdot\td{x}_i-\lu(\td{x}_i).\]
This implies
\be\label{cb0.7}
\td{\xi}_i\cdot\lt(\td{x}_i-x_0\rt)>\lu(\td{x}_i)-\lu(x_0)
\ee
and
\be\label{cb0.8}
\xi_0\cdot\lt(x_0-\td{x}_i\rt)>\lu(x_0)-\lu(\td{x}_i).
\ee
Using the same argument as in the Step 1 of Lemma \ref{cblem5}, we can see that the sequence $\{\td{x}_i\}$ is uniform bounded.
Choosing a convergent subsequence of $\{\td{x}_i\}$ which we still denote by $\{\td{x}_i\}$, we assume
$\{\td{x}_i\}\goto\td{x}_0$ as $i\goto\infty.$ Let $i\goto \infty,$ then from \eqref{cb0.7} and \eqref{cb0.8}
we conclude
\[\xi_0\cdot\lt(\td{x}_0-x_0\rt)=\lu(\td{x}_0)-\lu(x_0).\]
Applying Lemma \ref{cblem5} we get $\td{x}_0=x_0,$ which leads to a contradiction.
\end{proof}

\begin{lemm}
\label{cblem3}
Let $h^*$ be the Lengendre transform of $h.$ Then
$h^*(\xi)\leq\lu^*(\xi)$ for all $\xi\in D^*.$
\end{lemm}
\begin{proof}
For any $\xi\in\text{Int}D^*,$ there exist $x, y\in\R^n$ such that
$$h^*(\xi)=x\cdot\xi-h(x)$$
and
$$\lu^*(\xi)=y\cdot\xi-\lu(y).$$
Therefore
\begin{align*}
h^*(\xi)-\lu^*(\xi)&=x\cdot \xi-h(x)-y\cdot \xi+\lu(y)\\
&<(x-y)\cdot \xi+\lu(y)-\lu(x)<0,
\end{align*}
where the last inequality comes from $\lu(x)$ is strictly convex.
By continuity we have $h^*(\xi)\leq\lu^*(\xi)$ for all $\xi\in D^*.$
\end{proof}

\bigskip
\section{Construction of the convergence sequence}
\label{cs}
Following the step in \cite{WX}, after constructing $\lus,$ it's the time to solve the corresponding Dirichlet problem \eqref{dirichlet-ball}.
Unfortunately, the $\lus$ we constructed is only a $C^1$ function, thus the boundary value of the corresponding Dirichlet problem is only $C^1.$
However, for the sake of $C^2$ boundary estimates, we need at least $C^{1, 1}$ boundary data. A natural thing to do next is to smooth $\lus.$
Note that for technical reasons explained in Subsubsection \ref{sub1.2.2}, we need to obtain a supersolution of equation \eqref{dirichlet-ball} after the smoothing process. Therefore, in the following, we will use inf-convolution.

We will construct $\gjs$ in the following way. Let $\td{F}_j$ be a sequence of strictly convex, smooth open domain such that
$\td{F}_j\goto\td{F},$ $\td{F}_{j-1}\subsetneq\td{F}_j$ for each $j\geq 2,$ and $B_j(0)\subset D\lus(\td{F}_j).$
Here, $B_j(0)$ is a ball of radius $j$ centered at $0.$

Now, let  $\td{\lu}_j^*$ be the smooth extension of $\lus$ such that $\td{\lu}_j^*=\lus$ on $\td{F}_{j+1},$ $\td{\lu}_j^*$ is smooth outside $\td{F}_{j+2},$ and $\td{\lu}_j^*$ is a convex function defined on $\R^n.$ Let $\gjs(\xi)=\inf\limits_{\eta\in\mathbb{R}^n}\{\td{\lu}_j^*(\eta)+\frac{1}{2\e_j}|\xi-\eta|^2\},$ by Theorem 3.5.3 and
Corollary 3.3.8 of \cite{CS} we know that $\gjs\in C^{1, 1}_{loc}(\td{F}).$ Moreover, by Theorem 3.5.8 of \cite{CS},
well known results, and Lemma 4.2 of \cite{SI}, we know there exists
$\e_j>0$ so small that
\be\label{0.1}
0<\lus(\xi)-\gjs(\xi)<\frac{1}{j}\,\,\text{in}\,\,\td{F}_j,
\ee
\be\label{0.2}
|D\gjs(\xi)-D\lus(\xi)|<\frac{1}{j}\,\,\text{in}\,\,\td{F}_j,
\ee
and
\be\label{0.3}
F(\w\gas_{ik}\gjs_{kl}\gas_{lj})\leq\frac{1}{\binom{n}{k}^{\frac{1}{k}}}\,\,\text{in}\,\,\td{F}_j.
\ee
Here, $\w=\sqrt{1-|\xi|^2},$ $\gas_{ik}=\delta_{ik}-\frac{\xi_i\xi_k}{1+\w},$ $\gjs_{kl}=\frac{\T^2\gjs}{\T\xi_k\T\xi_l},$
and let
$\ka^*[\w\gas_{ik}\gjs_{kl}\gas_{lj}]=(\ka^*_1, \cdots, \ka^*_n)$ be the eigenvalues of the matrix $(\w\gas_{ik}\gjs_{kl}\gas_{lj}),$ then $F(\w\gas_{ik}\gjs_{kl}\gas_{lj})=\lt(\frac{\s_n}{\s_{n-k}}(\ka^*)\rt)^{\frac{1}{k}}.$

Since we can not find a literature that has the proof of \eqref{0.2}, we will include the proof here.
\begin{lemm}
\label{cslem0}
There exists $\e_j>0$ so that
\[|D\gjs(\xi)-D\lus(\xi)|<\frac{1}{j}\,\,\text{in}\,\,\td{F}_j.\]
\end{lemm}
\begin{proof}
Denote $g^*_\e(\xi):=\inf\limits_{\eta\in \R^n}\{\td{\lu}_j^*(\eta)+\frac{1}{2\e}|\xi-\eta|^2\}.$ Since $g^*_\e$ is $C^{1, 1},$
$D\lus(\cdot)$ is continuous (by Lemma \ref{cblem6}), and $\overline{\td{F}_j}$ is compact, we only need to show for a fixed
$\xi_0\in\td{F}_j,$ there exists an $\e>0$ small, such that
\[|Dg^*_\e(\xi_0)-D\lus(\xi_0)|<\frac{1}{j}.\]
Now assume
\[g^*_\e(\xi_0)=\lus(\eta_0)+\frac{1}{2\e}|\eta_0-\xi_0|^2.\]
We also notice that for any $\xi\in\R^n$
\[g^*_\e(\xi)\leq\lus(\xi-\xi_0+\eta_0)+\frac{1}{2\e}|\eta_0-\xi_0|^2.\]
Therefore, we obtain
\be\label{0.4}
g^*_\e(\xi_0)-g^*_\e(\xi)\geq \lus(\eta_0)-\lus(\xi-\xi_0+\eta_0).
\ee
Now let $\xi=\xi_0+hv,$ where $v\in\mathbb{S}^{n-1}$ is an arbitrary unit vector and $h>0$ small.
Then from \eqref{0.4} we get
\[\frac{g^*_\e(\xi_0)-g^*_\e(\xi_0+hv)}{h}\geq \frac{\lus(\eta_0)-\lus(\eta_0+hv)}{h}\geq-D\lus(\eta_0+hv)\cdot v.\]
Let $h\goto 0$ we derive
\[-Dg^*_\e(\xi_0)\cdot v\geq-D\lus(\eta_0)\cdot v.\]
By Proposition 2.1 of \cite{AF} we know $\eta_0\in B_{2\e M}(\xi_0),$ where $M$ is the Lipschitz constant of $\lus$ in $\td{F}_j.$
Moreover, by Lemma \ref{cblem6} we know $D\lus$ is continuous. Therefore, there exists $\delta=\delta(\e)>0$ such that
\[-Dg^*_\e(\xi_0)\cdot v\geq-D\lus(\xi_0)\cdot v-\delta\]
which is equivalent to
\[|Dg^*_\e(\xi_0)-D\lus(\xi_0)|\leq\delta.\]
Furthermore, we can see that as $\e\goto 0,$ $\eta_0\goto\xi_0,$ and $\delta\goto 0.$
Thus Lemma \ref{cslem0} is proved.
\end{proof}

Now, we consider the following Dirichlet Problem
\be\label{cs1.2}
\left\{
\begin{aligned}
F(\w\gas_{ik}\ujs_{kl}\gas_{lj})&=\frac{1}{\binom{n}{k}^{\frac{1}{k}}}\,\,\text{in $\td{F}_j$}\\
\ujs&=\vjs\,\,\text{on $\partial\td{F}_j,$}
\end{aligned}
\right.
\ee
where $\vjs=\gjs|_{\partial\td{F}_j}.$

By Section 5 of \cite{WX} we know that \eqref{cs1.2} is solvable, denote the solution by $\ujs.$ We also denote the Legendre transform
of $\ujs$ by $\uj.$ Applying the spacelike cutoff function $\td{\Psi}_{\F}$ constructed in Subsectrion
\ref{constructF}, following the argument of Section 6 in \cite{WX}, we conclude that there exists a subsequence of $\{\uj\}$
that converges to the desired entire solution $u$ satisfying \eqref{int1.0} and \eqref{int1.0'}. This completes the proof of
Theorem \ref{intthm1}.

\bigskip

\section{Uniqueness results for CMC hypersurfaces}
\label{uniqueness}
The uniqueness of the solutions to equations \eqref{int1.0} and \eqref{int1.0'} is an open problem. No results are known even
for CMC hypersurfaces. In this section, we will discuss about this problem.

By the Maclaurin's inequality, it's easy to see
if $\lu$ is a subsolution of equation $\s_k={n \choose k},$ then it is also a subsolution of equation $\s_l={n\choose l},$ for any $l<k.$
Thus, in view of Section \ref{construction}, for each $1\leq k<n,$ we can construct more than one subsolutions. It's natural to ask: can we obtain different solutions by using different subsolutions?

To make the problem simpler, we only consider the case when the Gauss map image is a half disc $\bar{B}_1^+$ and the perturbation $q\equiv 0.$
Due to technical reasons we can not answer this simplified question for the case when $k>1$. When $k=1$, we prove the following uniqueness result. This result makes us believe that, in general, our method cannot provide more then one solution to equations \eqref{int1.0} and \eqref{int1.0'}.

\begin{theo}
\label{unithm1}
Let $\M_u$ be an $n$-dimensional CMC hypersurfaces satisfying following conditions:

(1). $\sigma_1(\ka[\M_u])=n;$

(2).When $\frac{x}{|x|}\in\bar{\mathcal{B}}_+$ or $x_1\goto-\infty, |\bar{x}|\goto\infty,$ we have
$u(x)-V_{\bar{\mathcal{B}}_+}(x)\goto 0$ as $|x|\goto\infty$;

(3).$|u-V_{\bar{\mathcal{B}}_+}(x)|\leq C_0$ for any $x\in\R^n.$

Then $u$ is the standard semitrough $\z^1.$ Here $\bar{\B}_+=\dS^{n-1}\cap\{x_1\geq 0\}$ and $\bar{x}=(x_2, \cdots, x_n).$
\end{theo}
\begin{proof}
In this proof, we will denote
\[\Omega_R:=\{(x_1, \bar{x})\mid x_1\leq 0,\,\,|\bar{x}|\leq R\}.\]

\textbf{Step 1.} We will show that $\lim\limits_{x_1\goto-\infty}u(x_1, \bar{x})=\tilde{u}(\bar{x})$ uniformly in $\Omega_R.$
First, by condition (2) we get  $u(x_1, \bar{x})$ is strictly convex.
In view of condition (3) we know for any fixed $\bar{x}\in\R^{n-1},$ when $x_1\leq 0,$
$u(x_1, \bar{x})$ is monotone increasing and bounded. Therefore, $\lim\limits_{x_1\goto-\infty}u(x_1, \bar{x})$ exists and we denote it by
$\tilde{u}(\bar{x}).$

Next, we want to show $\lim\limits_{x_1\goto-\infty}u(x_1, \bar{x})=\tilde{u}(\bar{x})$ uniformly on $\Omega_R.$ If not, then there exists some $\eta_0>0$ such that for any $n\in\mathbb{N},$ there exist some $-i_n<-n$ and $b_n\in \bar{B}^{n-1}_R:=\{\bar{x}\in\R^{n-1}\mid |\bar{x}|\leq R\}$ such that
\be\label{uni1}
u(-i_n, b_n)-\tilde{u}(b_n)>\eta_0>0.
\ee
Without loss of generality, we assume $\{b_n\}_{n=1}^\infty$ converges to $b_0.$
Since $|Du|\leq 1$, when $n>N_0$ we have, $|u(-i_n, b_n)-u(-i_n, b_0)|<\frac{\eta_0}{4}$
and $|\tilde{u}(b_n)-\tilde{u}(b_0)|<\frac{\eta_0}{4}.$ Combining with \eqref{uni1} we get,
$u(-i_n, b_0)-\tilde{u}(b_0)>\frac{\eta_0}{2}$ for all $n>N_0.$ This contradicts the assumption that
$\lim\limits_{x_1\goto-\infty}u(x_1, \bar{x})=\tilde{u}(\bar{x}).$

\textbf{Step 2.} In this step, we will show that $|Du|$ is uniformly bounded away from 1 on $\Omega_R.$

We will prove by contradiction. Assume the claim is not true, then we would have a sequence $\{(-i_n, b_n)\}_{n=1}^\infty\subset\Omega_R$
such that $|Du(-i_n, b_n)|\goto 1$ as $n\goto\infty,$ i.e., $-e_1=(-1, 0, \cdots, 0)$ is a light-like direction of $u(x).$ By condition (3)
we know $|u(-i_n,b_n)|$ is bounded for any $n\in\mathbb{N}$. Therefore, as $n\goto\infty$ $|Du(-i_n,b_n)|\goto 0.$ This leads to a contradiction.

which leads to a contradiction.

\textbf{Step 3.} In this step, we will show that $u=\z^1,$ where $\z^1$ is the standard semitrough constructed in Section \ref{se}.

Let $u^i(x)=u(x_1-i, \bar{x}), i\in\mathbb{N},$ and denote $U_R=\{x\in\R^n\mid -R\leq x_1\leq 0,$ $|\bar{x}|\leq R\}$. Then by step 1 and 2 we know that
$|u^i(x)|$ and $|Du^i(x)|$ are uniformly bounded in $U_R.$ Moreover, by \cite{Tre} we also know that $|D^2u^i|$ is uniformly bounded in $\R^n.$
Therefore, we conclude that $u^i(x_1, \bar{x})\goto\tilde{u}(\bar{x})$ converges smoothly on any compact set $K\subset\{(x_1, \bar{x})\mid x_1\leq 0\}$
and $\tilde{u}(\bar{x})$ satisfies $\sigma_1(\ka[\M_{\tilde{u}}])=n.$

In view of our assumption that
\[|u(x)-V_{\bar{\mathcal{B}}_+}(x)|\goto 0\,\,\mbox{as $x_1\goto-\infty$ and $|\bar{x}|\goto\infty$},\]
we obtain
\[|\tilde{u}(\bar{x})-V_{\bar{\mathcal{B}}_+}(0,\bar{x})|\goto 0\,\,\mbox{as $\bar{x}\goto\infty.$}\]
Thus, $\M_{\tilde{u}}$ is an $(n-1)$-dimensional CMC hypersurface satisfying $\sigma_1(\ka[\M_{\tilde{u}}])=n.$ Moreover, the Gauss map image of
$\M_{\tilde{u}}$ is $B_1^{n-1}:=\{x\in\R^{n-1}\mid |x|\leq 1\}.$ By the maximum principle we know that $\tilde{u}(\bar{x})=\sqrt{\lt(\frac{n-1}{n}\rt)^2+|\bar{x}|^2}.$
From the above discussions we conclude that as $|x|\goto +\infty, u(x)\goto \z^1(x)$.
Applying the maximum principle again we get $u\equiv \z^1.$
\end{proof}

\end{document}